\documentclass[11pt]{amsart}
\usepackage{a4wide}
\usepackage{amsmath,amssymb,amsfonts,amsthm}

\usepackage{xcolor}

\usepackage{tikz}
\usetikzlibrary{decorations.pathreplacing}
\usepackage{cancel}

\usepackage[colorlinks=true,urlcolor=blue,
citecolor=red,linkcolor=blue,linktocpage,pdfpagelabels,
bookmarksnumbered,bookmarksopen]{hyperref}

\theoremstyle{plain}

\newtheorem{theorem}{Theorem}[section]
\newtheorem{lemma}[theorem]{Lemma}
\newtheorem{proposition}[theorem]{Proposition}
\newtheorem{corollary}[theorem]{Corollary}

\theoremstyle{definition}
\numberwithin{equation}{section}

\def\til{~}

\def\Om{\Omega}
\def\R{\mathbb{R}}

\def\Z{\mathbb{Z}}
\def\N{\mathbb{N}}
\def\dist{\textup{dist}}

\def\H{\mathcal{H}}

\def\T{\mathbb{T}}

\renewcommand{\div}{\mathrm{div}}

\newcommand{\e}{\varepsilon}

\newcommand{\okappa}{\overline\kappa}

\newcommand{\pa}{\partial}

\newcommand{\medint}{-\kern -,375cm\int}
\newcommand{\medintinrigo}{-\kern -,315cm\int}

\def\beq{\begin{equation}}
\def\eeq{\end{equation}}

\title[The Asymptotics of the 2D Area-Preserving Mean Curvature Flow]{The Asymptotics of the Area-Preserving Mean Curvature and the Mullins-Sekerka Flow in Two Dimensions}

\author{Vesa Julin}\address{Matematiikan ja Tilastotieteen Laitos, Jyv\"askyl\"an Yliopisto, Finland}\email{vesa.julin@jyu.fi}

\author{Massimiliano Morini}\address{Dipartimento di Scienze Matematiche Fisiche e Informatiche, Universit\`a di Parma, Italy}\email{massimiliano.morini@unipr.it}

\author{Marcello Ponsiglione}\address{Dip. di Matematica, Univ. Roma-I ''La Sapienza'', Roma, Italy} \email{ponsigli@mat.uniroma1.it}

\author{Emanuele Spadaro}\address{ Dip. di Matematica, Univ. Roma-I ''La Sapienza'' Roma, Italy}\email{spadaro@mat.uniroma1.it}

\begin{document}

\begin{abstract} 
We provide the first general result for the asymptotics of the area preserving mean curvature flow in two dimensions showing that   flat flow solutions, starting from any bounded set of finite perimeter, converge with exponential rate to a finite union of  equally sized disjoint disks.  A similar  result is established also for the periodic two-phase Mullins-Sekerka flow. \end{abstract}

\maketitle

\section{Introduction}
In this paper we address the long-time behaviour of two physically relevant area preserving nonlocal geometric flows in the plane: the area-preserving mean curvature and the Mullins-Sekerka flow.  

We start by recalling that a smooth flow of sets $(E_t)_{t\in [0,T)}\subset \R^2$, for some $T>0$, is a solution to the area preserving mean curvature flow if it satisfies

\beq \label{eq:VMCF}
V_t = -\kappa_{E_t} +  \okappa_{E_t} \quad\text{on }\pa E_t\,,
\eeq
where $V_t$ denotes the normal velocity, $\kappa_{E_t}$ the curvature and  $\okappa_{E_t}:=\medintinrigo_{\pa E_t}\kappa_{E_t}\, d\H^1 $ the integral average of the curvature of the evolving boundary $\pa E_t$. Such a  geometric flow has been proposed in the physical literature as a model for coarsening phenomena. For example, one can consider systems that,
after a first relaxation time, can be described by two subdomains of nearly pure phases far from equilibrium,
evolving in a way to decrease the total interfacial length between the phases while keeping their area constant  (for the physical background see \cite{CRCT95, MuSe13,TW72,W61}). An important feature of the flow is that it can be regarded  as a gradient flow of the perimeter with respect to a suitable (formal) $L^2$-type Riemannian structure.

The second geometric evolution we consider,  the  two-phase Mullins-Sekerka flow in the flat torus $\T^2$, is governed by the law 
\beq
\label{eq:Mull-Sek}
\begin{cases}
V_t= [\pa_{\nu_t}u_t] & \text{on $\pa E_t$,}\\
-\Delta u_t=0 & \text{in $\T^2 \setminus \pa E_t$,}\\
u_t=\kappa_{E_t} & \text{on $\pa E_t$},
\end{cases}
\eeq
where $\nu_t$ denotes the external normal to $\partial E_t$, $[\pa_{\nu_t}u_t]$ denotes  the jump of the normal derivative of $u_t$ at $\pa E_t$, i.e.,  $[\pa_{\nu_t}u_t]:=\pa_{\nu_t}u^+_t-\pa_{\nu_t}u^-_t$, with  $u^+_t$ and $u^-_t$ denoting the restrictions of $u_t$ to $\T^2 \setminus E_t$ and $E_t$ respectively, and $\kappa_{E_t}$ is as before the curvature of the evolving boundary. Let us notice that the choice of the flat torus $\T^2$ instead of a bounded domain $\Om$ is made to avoid in the first place boundary effects.
The Mullins-Sekerka flow is a nonlocal generalization of the mean curvature flow arising from physics. It can be seen as a quasistatic variant of the Stefan problem  and it was originally proposed as an isotropic model for solidification and liquefaction phenomena when the specific heat is negligible, see \cite{MS}. Moreover, it arises as a singular limit of the Cahn-Hilliard equation, see \cite{AlBaCh, Pego}. Common features with \eqref{eq:VMCF} are the area preserving character and the gradient flow structure (this time with respect to a suitable $H^{-1}$-Riemannian structure).

It is well-known that, in general, smooth solutions of \eqref{eq:VMCF} may develop singularities  in finite time, such as disappearance  and coalescence of components,  pinch-offs and curvature blow-up, even in two dimensions (see for instance \cite{EI, M, MaSim}). The same can be expected for the flow \eqref{eq:Mull-Sek}. 
  The possible singular behaviour of \eqref{eq:VMCF} and \eqref{eq:Mull-Sek} is even wilder than that of the unconstrained mean curvature flow, due to their nonlocal character  and the subsequent lack of a comparison principle.
 Thus, for a well defined global-in-time evolution one has to introduce  suitable notion of weak solution which is  capable of handling  singularities, changes in topology and, possibly, rough initial data.  This is a well-established feature of curvature flows, and for several geometric motions, different definitions of weak solutions have been introduced in the literature.

Due to the lack of a comparison principle and based on the underlying gradient flow structure, a natural choice for \eqref{eq:VMCF} and \eqref{eq:Mull-Sek} is the minimizing movement approach proposed for the mean curvature flow indepentently by Almgren, Taylor and Wang \cite{ATW} and by Luckhaus and Sturzenhecker \cite{LS}, and adapted to the volume constrained case in \cite{MSS}. Note that Luckhaus and Sturzenhecker \cite{LS} introduce a similar variational scheme for \eqref{eq:Mull-Sek} as well, see also \cite{Rog} where the same scheme is further analyzed. We recall that the minimizing movement method is based on implicit time-discretization  and recursive minimization of suitable incremental problems. 
The limiting time-continuous evolutions constructed in this way are usually referred to as \emph{flat flows}. We refer to Sections~\ref{sec:3}~and~\ref{sec:4} for the precise definition of flat flow solution of \eqref{eq:VMCF} and \eqref{eq:Mull-Sek}, respectively.

Once  a global-in-time weak solution has been constructed, it is a natural problem to investigate its asymptotics. The focus of the paper is the long-time behaviour of flat flows in two dimensions. Previous results on the long-time convergence of volume preserving flows are mostly confined to the case of smooth solutions starting from specific classes of initial regular sets, see for instance \cite{ES, Hui, Joonas} for the volume preserving mean curvature flow and \cite{AFJM, Chen, ES98, GarRau} for the Mullins-Sekerka flow.  For less general initial data, 
the long time behaviour of the volume preserving mean curvature flow starting from convex and star-shaped sets (see \cite{BCCN, KK})  has been characterized only up to (possibly diverging in the case of \cite{BCCN}) translations. 
Finally, concerning flat flow solutions the most general result is due to \cite{JN} where the asymptotic convergence to  finitely many disjoint balls is proven in two and three dimensions for arbitrary bounded initial sets of finite perimeter, but only up to (possibly diverging) translations and without a convergence rate. 
 
 In our main result we are able to rule out translations and we provide in two dimensions the first full convergence result for the asymptotics of the area preserving mean curvature and the Mullins-Sekerka flow.
We show that every flat flow solutions of \eqref{eq:VMCF} starting from \emph{any} set of finite perimeter asymptotically converge, with \emph{exponential rate}, to a finite disjoint union of (possibly tangent) equally sized discs.
Under the additional assumption that the perimeter of the initial set is smaller than $2$, we establish a similar result also for \eqref{eq:Mull-Sek}). Note that such an additional condition is assumed for simplicity to rule out lamellae as possible limiting sets (see Sections~\ref{sec:4} for further details). We refer to the next section for the precise statements.  

Let us finally mention that the analysis  of this paper extends in two dimensions the results proven in \cite{MoPoSpa} (see also \cite{DeGKu} for related results in the flat torus)  for the discrete minimizing movements of the volume preserving mean curvature flow to the time-continuous limiting evolutions. 
 
 \subsection{Statement of the main results}\label{sec:1.1}

In the previous work \cite{MoPoSpa} three of the authors prove that in all dimensions the discrete  approximate volume preserving mean curvature flow converges exponentially fast to a disjoint union of  balls with equal size. This is the optimal convergence result but it leaves open the question of the convergence of the limiting flat flow. On the other hand, in \cite{JN} the first author and Niinikoski prove that the limiting flat flow converges in low dimensions $\R^2$ and $\R^3$ to a disjoint union of balls, up to possible translations of the components. Again this result does not prove the full convergence nor does it provide any rate of convergence. In both papers  it was observed that a key technical issue is to prove a quantitative version of the Alexandrov theorem, which in the classical form states that the only compact smooth hypersurfaces with constant mean curvature are union of spheres. In this paper we develop this idea further and observe that we may prove a geometric inequality, very much related to the  quantitative Alexandrov theorem, which implies the full convergence of the flow and also gives the exponential rate of convergence.    

There has been a lot of recent research on generalizations and quantifications of the Alexandrov theorem. We refer to \cite{Ci} for an overview of this challenging problem, and mention the works \cite{DM, DMMN, RKS} on the  characterization of critical sets of the isoperimetric problem and \cite{CM,CV, KM} on quantification of the Alexandrov theorem.

We state our quantitative version in a form that is suitable for the study of the equation \eqref{eq:VMCF}. 
We denote the length of the boundary or more generally the  perimeter of a set $E$ by $P(E)$ and  by $|E|$ its area. We also denote by $P_d= 2\sqrt{\pi m d}$ the perimeter of the union of $d$ disks with total area $m$.  Our first result reads as follows. 
\begin{theorem}\label{cor:2DAle}
  Let $m, M>0$ and let $E\subset\R^2$ be a bounded open set of class $C^2$, with $|E|=m$ and  $P(E)\leq M$. Then there exists a constant $C(m,M)>0$ such that 
\[
\min_{d \in \N}|P(E) - P_d|  \leq C \|\kappa_E - \okappa_E \|_{L^2(\pa E)}^2.
\]
Moreover, if $\delta_0>0$ and $d \in  \N$, are such that $P_{d} \leq P(E) \leq P_{d+1} - \delta_0$, then it holds 
\begin{equation}\label{e.coroAlex}
P(E) - P_d  \leq C_0 \|\kappa_E -\okappa_E \|_{L^2(\pa E)}^2,
\end{equation}
with $C_0= C_0(m,M,\delta_0)$. 
\end{theorem}

The novelty of the above result is that on the right-hand-side we have quadratic dependence on the curvature which is clearly  optimal. One may compare this result to the quantification of the Willmore energy \cite{RS} or to the optimal quantitative isoperimetric inequality \cite{FMP}, which both have similar scaling. The inequalities in Theorem \ref{cor:2DAle} are geometric and do not measure how close the set $E$ is to the union of disks. In the planar case the closeness of $E$ to the union of disks is proven in \cite{FJM} (see also Proposition \ref{prop:2DAle} in Section 2). The above result is proven in the planar case but it could be true also in higher dimensions.  

As we already mentioned, the motivation for the geometric inequality in Theorem \ref{cor:2DAle} is the proof of the asymptotic convergence of the area-preserving mean curvature flow equation \eqref{eq:VMCF}. 

\begin{theorem} \label{thm2}
Let $\{E(t)\}_{t\geq 0}$ be an area-preserving flat flow for \eqref{eq:VMCF} starting from a bounded set of finite perimeter $E(0)\subset\R^2$.
Then, there exist $d\in \N$ disjoint open disks in the plane $D_r(x_1), \ldots, D_r(x_d)$, with $\pi r^2 d = |E(0)|$, and there exists a constant  $C>1$ such that, setting $ E_\infty= \bigcup_{i=1}^d D_r(x_i)$, it holds 
\beq\label{expo0}
\sup_{x\in E(t) \Delta E_\infty}\dist(x, \pa E_\infty)+
|P(E(t))-P(E_\infty)|\leq C e^{-\frac{t}{C}}
\eeq
for all $t\geq0$. 
\end{theorem}

The above theorem gives the full characterization and quantitative speed of convergence of the equation \eqref{eq:VMCF}. We expect the result to be sharp, in the sense that the flow may, indeed, converge to a union of tangent disks. In \cite[Theorem 1.4]{FJM} the authors consider the case when the initial set is a union of two ellipses and show that the equation \eqref{eq:VMCF} is well defined and smooth for all times and converges to two tangent disks. In particular, we may not improve the Hausdroff convergence in Theorem \ref{thm2}  to  $C^1$-convergence of the sets. The exponential convergence rate is optimal but we note that the flow may in fact converge to the limiting disks also in finite time. This is the case when we consider as an initial set a union of two disks $D_1, D_2$, which are far apart and $D_2$ is much smaller than $D_1$. Then along the flow the larger disk grows and the smaller one shrinks until it vanishes completely and the flow reaches its equilibrium state in finite time. The same phenomenon occurs when $D_2$ is only slightly smaller than $D_1$ but the time to reach the  equilibrium state tends to infinity when the size of $D_2$ gets closer to the size of $D_1$. This shows that we cannot bound  the constant $C$ by a universal constant, but it may depend on the initial set in a rather complicated way.

We note that our method can be also used to study asymptotic behavior of other geometric flows, and to emphasize this we also address the asymptotics of the two-phase Mullins-Sekerka flow \eqref{eq:Mull-Sek}. To avoid boundary effects  we consider periodic conditions and set the problem in the flat torus $\T^2$ and, as a further simplification, we consider initial configurations with perimeter smaller than that of the single lammella (alternatively, we can think that the size of the torus is big enough compared to the perimeter of the initial set).  

The main result is the following. We denote the perimeter of a set $E$ in the flat torus by $P_{\T^2}(E)$. 
\begin{theorem} \label{thm3}
Let $\{E(t)\}_{t\geq 0}$ be a flat flow solution to the Mullins-Sekerka flow  \eqref{eq:Mull-Sek} in the flat torus $\T^2$ starting from a  set of finite perimeter $E(0)\subset\T^2$, with $P_{\T^2}(E) < 2$.
Then, there exist $d\in \N$ disjoint open disks $D_r(x_1), \ldots, D_r(x_d)$, with $\pi r^2 d = |E(0)|$, and there exists a constant  $C>1$ such that it holds 
\[
|E(t) \Delta E_\infty| + 
|P(E(t))-P(E_\infty)|\leq C e^{-\frac{t}{C}}
\]
for all $t\geq 0$, where   $ E_\infty$ either coincides with $\bigcup_{i=1}^d D_r(x_i)$ or with its complement in $\T^2$.

\end{theorem}

The proof of Theorem \ref{thm3} is similar to that of the previous theorem. We use Theorem \ref{cor:2DAle} and a result by Sch\"atzle  \cite{Sch} to obtain a functional inequality (see Corollary \ref{coro:2Dale2}), which is in the spirit  of the quantitative Alexandrov theorem, stated now in terms of the potential $u_t$.

We remark that one could also consider the one-phase model for the Mullins-Sekerka as in \cite{CL} in the whole $\R^2$ and expect the above convergence to hold also in this case. We also expect the convergence of the sets in Theorem \ref{thm3} to hold with respect to Hausdorff distance but we do not prove it here. 

\subsection{Structure of the paper}

Section~\ref{sec:2} is purely geometric and in Proposition~\ref{prop:2DAle} we prove our quantitative version of the Alexandrov theorem which then implies  Theorem~\ref{cor:2DAle} as a corollary. In Section~\ref{sec:3} we first introduce the incremental minimization problem for the minimizing movements scheme, and recall some basic results related to its minimizers. Then we recall the construction of the flat flow and  give the proof of Theorem~\ref{thm2} at the end of the section. In Section~\ref{sec:4} we introduce the incremental  minimization problem and the flat flow for the Mullins-Sekerka equation. We then state  and prove in Proposition~\ref{prop:2DAle2} a crucial functional inequality  which is related to Proposition~\ref{prop:2DAle}. The section concludes with the proof of Theorem \ref{thm3}.

\medskip



\section{A sharp quantitative Alexandrov theorem in two-dimensions}\label{sec:2}

Let us first recall that for measurable sets $E\subset \R^2$, the perimeter is defined by
\[
P(E) := \sup \Big{\{} \int_E \div X \, dx : X \in C_c^1(\R^2,\R^2) , \, \| X\|_{L^\infty} \leq 1 \Big{\}}.
\]
If $P(E) < \infty$ we say that $E$ is a set of finite perimeter. We also recall that if $E$ is regular enough, say a domain with Lipschitz boundary, then $P(E) = \H^1(\pa E)$.  For the general properties of sets of finite perimeter we refer to the monographs \cite{AFP, MaggiBook}.

In the following we fix the prescribed area $m>0$ of a set $E$ and a constant $M>0$ representing an upper bound for the perimeter of $E$. For $d\in \N$ we denote by $P_d$ the perimeter of any union of $d$ disjoint disks with equal areas $m/d$, i.e.,
\[
P_d:=2\sqrt{\pi m d}\,.
\]
For a set of $E \subset \R^2$ of class $C^2$ we denote by $\kappa_E$ its curvature (with the sign defined so that $\kappa_E$ is positive for convex sets) and we set 
\[
\okappa_E:=\medint_{\pa E}\kappa_E\, d\H^1= \frac1{\H^1(\pa E)} \int_{\pa E} \kappa_E\, d\H^1.
\]
In \cite{FJM} it is proven that if  $E \subset \R^2$ is a set of class $C^2$ with area $|E|=m$ and $\|\kappa_E-\okappa_E \|^2_{L^1(\pa E)} \leq \e_0$, for $\e_0$ small enough,  then $E$ is $C^1$-diffeomorphic to a disjoint union of disks $D_1, \dots, D_d$ and it holds 
\[
|P(E) - P_d| \leq C \|\kappa_E-\okappa_E \|_{L^1(\pa E)}\,.
\]
Our first result improves the above inequality by showing that a similar estimate holds with quadratic right-hand side, which is the optimal scaling of the quantitative Alexandrov theorem.
We also consider $L^2$-norms as this is more natural in our variational framework.  We state this in the following proposition. 
\begin{proposition}\label{prop:2DAle}
Let $m, M>0$.
There exist $\e_0=\e_0(m,M)\in (0,1)$ and $C_0=C_0(m, M)>1$ with the following property:  Let $E\subset\R^2$ be a bounded open set of class $C^2$, with $|E|=m$ and  $P(E)\leq M$, such that  $\|\kappa_E-\okappa_E \|_{L^2(\pa E)}\leq \e_0$. Then 
$E$ is diffeomorphic to a union of $d$ disjoint disks $D_1$, \dots, $D_{d}$, with equal areas $m/{d}$ and $\dist(D_i, D_j)>0$ for $i\neq j$, and
\beq\label{eq:geoineq}
|P(E)- P_{d}|\leq C_0\|\kappa_E-\okappa_E \|^2_{L^2(\pa E)}\,.
\eeq
Moreover, $d$ is bounded from above by a constant depending only on $m,\,M$. 

Finally, for $\e_0$ sufficiently small, the boundary of every connected component of the set $E$ can be parametrized as a normal graph over one of the discs $D_i$ with $C^{1,\frac12}$ norm of the parametrization vanishing as $\e_0\to 0$.
\end{proposition}
\begin{proof}
Let $E$ be as in the statement and let $E_1$,\dots, $E_{d}$ be the collection of its connected components. For each component $E_i$ we denote by $\Gamma_i$ the outer component of $\pa E_i$ and by $\hat E_i$ the bounded region  enclosed by $\Gamma_i$, i.e., the set obtained by filling the ``holes'' of $E_i$.

We split the proof into several steps. 
Notice that in what follows $m_0\in (0,1)$ and $M_0>1$ will denote ``universal'' constants, i.e, constants depending only on $m, M$, which may change from line to line. 

\medskip

\noindent{\bf Step 1.} We claim that 
\beq\label{eq:step1}
|\hat E_{\bar k}|\geq m_0\quad\text{for some }{\bar k}\in \{1, \dots, d\}.
\eeq
Indeed, by translating the components if necessary we may assume that $\dist (\hat E_i, \hat E_j)>\sqrt 2$.
Setting $Q:=(0,1)\times(0,1)$, we may use  \cite[Lemma\til2.1]{MoPoSpa}
to infer that there exist $z\in \Z^2$ such that 
$$
|E\cap (z+Q)|\geq c \min\Big\{\frac{m^2}{M^2}, 1\Big\}\,,
$$
with $c>0$ a universal constant. Since $z+Q$ can only intersect one component $\hat E_i$, the claim follows. 

\medskip

\noindent{\bf Step 2.} We claim that
\beq\label{eq:step2}
|\okappa_E|\leq M_0\,.
\eeq
 To this aim, note that by the Isoperimetric Inequality and by \eqref{eq:step1}, we have
\beq\label{st2-2}
\H^1(\pa \hat E_{\bar k})\geq 2\sqrt{\pi m_0}\,.
\eeq
Now, 
\[
\int_{\pa \hat E_{\bar k}}\Big|\kappa_{E}-\frac{2\pi}{\H^1(\pa \hat E_{\bar k})}\Big|^2\, d\H^1\leq \int_{\pa E}|\kappa_{E}-\okappa_E |^2\, d\H^1\leq \e_0^2\,,
\]
where we used the simply connectedness of $\hat E_{\bar k}$ and Gauss-Bonnet Theorem to get $\okappa_{\hat E_{\bar k}}=2\pi/\H^1(\pa \hat E_{\bar k})$. In turn, 
\[
\begin{split}
\frac1{\H^1(\pa \hat E_{\bar k})}&\Big|2\pi -\H^1(\pa \hat E_{\bar k})\okappa_{E}\Big|^2
=\int_{\pa\hat E_{\bar k}}\Big|\frac{2\pi}{\H^1(\pa \hat E_{\bar k})}-\okappa_E\Big|^2\, d\H^1 \\
&\leq 2\int_{\pa\hat E_{\bar k}}\Big|\kappa_{E}- \frac{2\pi}{\H^1(\pa \hat E_{\bar k})}\Big|^2\, d\H^1+ 2\int_{\pa E}\Big|\kappa_{E}- \okappa_E\Big|^2\, d\H^1\leq 4\e_0^2\leq 4\,.
\end{split}
\]
Hence $\Big|2\pi -\H^1(\pa \hat E_{\bar k})\okappa_{E}\Big| \le 2 \sqrt{\H^1(\pa \hat E_{\bar k})}$, so that, using also \eqref{st2-2},
$$
2\sqrt{\pi m_0}|\okappa_E|\leq 2\pi+|2\pi -\H^1(\pa \hat E_{\bar k})\okappa_{E}|\leq 
2\pi+2\sqrt{\H^1(\pa \hat E_{\bar k})}\leq 2\pi(1+\sqrt M)\,,
$$
and the claim follows. 

\medskip

\noindent{\bf Step 3.} We claim that 
\beq\label{eq:step3}
\H^1(\Gamma)\geq m_0 \qquad\text{for any component }\Gamma\text{ of }\pa E\,.
\eeq
Indeed, using again Gauss-Bonnet Theorem, 
\[
\begin{split}
M |\okappa_E|^2+1&\geq \H^1(\Gamma)|\okappa_E|^2+\e_0^2\geq 
\H^1(\Gamma)|\okappa_E|^2+\int_\Gamma|\kappa_E-\okappa_E|^2\, d\H^1\\
&\geq\frac12\int_\Gamma|\kappa_E|^2\, d\H^1\geq\frac1{2\H^1(\Gamma)}
\Big(\int_\Gamma \kappa_E\Big)^2=\frac1{2\H^1(\Gamma)}4\pi^2\,,
\end{split}
\]
and the claim follows taking into account \eqref{eq:step2}.

\medskip

\noindent{\bf Step 4.} We claim that if $\e_0$ is sufficiently small, then $E$ has $d\leq M_0$ connected components which are simply connected.

We argue by contradiction. Suppose there exists a connected component $E_i$ which is not simply connected. Then there exists a component $\Gamma\subset \pa E$ contained in $\hat E_i$ such that $\int_{\Gamma} \kappa_{E} \, d \H^1 = -2 \pi$.
We observe that then it holds
\[
\medint_\Gamma \kappa_E d \H^1 = - \frac{2\pi}{\H^1(\Gamma)} \qquad\text{and} \qquad 
\medint_{\pa \hat E_i} \kappa_E d \H^1 =  \frac{2\pi}{\H^1(\pa \hat E_i)}
\]
and therefore
\[
\int_\Gamma \Big|\kappa_E + \frac{2\pi}{\H^1(\Gamma)} \Big|^2 d\H^1 + 
\int_{\pa\hat E_i}\Big|\kappa_E-\frac{2\pi}{\H^1(\pa \hat E_i)}\Big|^2 d\H^1 \leq 2\int_{\pa E} |\kappa_E - \okappa_E|^2 d \H^1\leq 2 \e_0^2.
\]
We then infer that by \eqref{eq:step3}
\begin{align*}
\frac{16\pi^2}{M^2} &\leq 
\left\vert \frac{2\pi}{\H^1(\Gamma)}+\frac{2\pi}{\H^1(\pa \hat E_i)} \right\vert^2 
\leq 
2\left\vert \okappa_E+\frac{2\pi}{\H^1(\Gamma)}\right\vert^2 +  2\left\vert\okappa_E-\frac{2\pi}{\H^1(\pa \hat E_i)} \right\vert^2 \\
& \leq 2\medint_\Gamma \left\vert \kappa_E+\frac{2\pi}{\H^1(\Gamma)}\right\vert^2 d \H^1 + 2\medint_{\pa \hat E_i} \left\vert\kappa_E-\frac{2\pi}{\H^1(\pa \hat E_i)} \right\vert^2 d\H^1\leq \frac{4\e_0^2}{m_0^2}.
\end{align*}
Therefore, for $\e_0$ sufficiently small we reach a contradiction.  

Every component of  $E$ is  thus simply connected and by \eqref{eq:step3} their perimeter is bounded from below. Therefore the number $d$ of the components is bounded from above $d \leq M_0$.  Note that in particular $\okappa_E = \frac{2\pi d}{\H^1(\pa E)}$.

\medskip

\noindent{\bf Step 5.} Let us show that if $\e_0$ is sufficiently small, then each connected component $E_i$ is a nearly spherical set, parametrized over a disks $D_{r_i}(x_i)$ with $|D_{r_i}(x_i)|= |E_i|$  and the $C^{1,\frac12}$ norm of the parametrization is infinitesimal with $\e_0 \to 0$. 

We adapt the argument of \cite[Lemma 3.2]{FJM}.
Let us fix a component $E_i$ and denote its perimeter by $l_i$, i.e. $H^1(\pa E_i) = l_i$. By Gauss-Bonnet it holds $\okappa_{E_i} = \frac{2 \pi}{l_i}$.
Since the boundary $\pa E_i$ is connected we may parametrize it by a unit speed curve  $\gamma : [0, l_i] \to \R^2$,  $\gamma(s) = (x(s),y(s))$ with counterclockwise orientation. Define $\theta(s) := \int_{0}^s \kappa_{E_i}(\gamma(\tau)) \, d \tau $ so that 
$\theta(0) = 0$ and $\theta(l_i) = 2 \pi$.  Then, 
for every $0 \leq s_1 < s_2 \leq l_i$, it holds by H\"older's inequality
\beq \label{eq:step5-2}
\begin{split}
|\theta(s_2) - s_2\okappa_E- (\theta(s_1)-s_1\okappa_E)|&\leq
\int_{s_1}^{s_2} |\kappa_E-\okappa_E|\\
&\leq \|\kappa_E -\okappa_E\|_{L^2(\pa E)} |s_2- s_1|^{\frac12}\\
&\leq \e_0 |s_2- s_1|^{\frac12}.
\end{split}
\eeq
In particular, applying \eqref{eq:step5-2} to $s_1=0$ and $s_2=s\in[0,l_i]$ generic, we get
\begin{equation}\label{eq:step5-2.5}
\big|\theta(s) - s\okappa_E\big| \leq  M_0 \|\kappa_E-\okappa_{E} \|_{L^2(\pa E)} \leq M_0 \e_0,
\end{equation}
and for $s_2 = l_i$ it yields  
\beq\label{eq:step5-3}
\left|2 \pi - l_i\okappa_E\right| \leq M_0\|\kappa_{E} -\okappa_{E}\|_{L^2(\pa E)} \leq M_0 \e_0. 
\eeq 
By possibly rotating the set $E_i$ we have 
\[
x'(s) = - \sin \theta(s) \qquad \text{and} \qquad y'(s) = \cos \theta(s) \quad \text{ for all } s\in(0,l_i) .
\]
We obtain by \eqref{eq:step5-2.5} and \eqref{eq:step5-3} that 
\beq
\label{eq:step5-4}
\Big|x'(s)  +  \sin \big(\frac{2\pi s}{l_i}\big)  \Big| + \Big|y'(s)  -  \cos \big(\frac{2\pi s}{l_i}\big)  \Big| \leq  M_0\|\kappa_{E} -\okappa_{E}\|_{L^2(\pa E)}\leq M_0 \e_0
\eeq
for all $s \in [0,l_i]$. Integrating \eqref{eq:step5-4} we deduce that  there are numbers $a$ and $b$ such that 
\begin{align}
\label{eq:step5-5}
\Big|x(s) - a  -  \frac{l_i}{2\pi} \cos \big(\frac{2\pi s}{l_i}\big)  \Big| &+ \Big|y(s) - b  -  \frac{l_i}{2\pi} \sin \big(\frac{2\pi s}{l_i}\big)  \Big|\notag\\
&\leq M_0\|\kappa_{E} -\okappa_{E}\|_{L^2(\pa E)}  \leq M_0 \e_0
\end{align}
for all $s \in [0,l_i]$. 
We set $x_i = (a,b)$ and note that from \eqref{eq:step5-5} we infer that
\begin{equation}\label{eq:step5.5.5}
D_{\frac{l_i}{2\pi}-M_0 \|\kappa_{E} -\okappa_{E}\|_{L^2(\pa E)}}(x_i) \subset E_i \subset 
D_{\frac{l_i}{2\pi}+M_0 \|\kappa_{E} -\okappa_{E}\|_{L^2(\pa E)}}(x_i).
\end{equation}
In particular, if $r_i$ is chosen in such a way that $|E_i|=\pi r_i^2 =  |D_{r_i}(x_i)|$, then \eqref{eq:step5.5.5} yields
\begin{equation}\label{eq:step5-5.6}
\frac{l_i}{2\pi}-M_0 \|\kappa_{E} -\okappa_{E}\|_{L^2(\pa E)} \leq r_i \leq \frac{l_i}{2\pi}+M_0 \|\kappa_{E} -\okappa_{E}\|_{L^2(\pa E)},    
\end{equation}
and 
\begin{align}
\label{eq:step5-5.7}
\Big|x(s) - a  -  r_i \cos \big(\frac{2\pi s}{l_i}\big)  \Big| &+ \Big|y(s) - b  -  r_i \sin \big(\frac{2\pi s}{l_i}\big)  \Big|\notag\\
&\leq M_0\|\kappa_{E} -\okappa_{E}\|_{L^2(\pa E)}  \leq M_0 \e_0
\qquad \forall\;s\in[0,l_i].
\end{align}
By \eqref{eq:step5-4} and \eqref{eq:step5-5.7} the boundary of the component $E_i$ is parametrized by a small perturbation of the boundary of the disc $\pa D_{r_i}(x_i)$ given by $c:[0,l_i]\to \R^2$
with $c(s) = r_i (\cos(\frac{2\pi s}{l_i}), \sin(\frac{2\pi s}{l_i}))$:
\begin{align}\label{eq:step5-5.8}
& \gamma(s)= c(s)+\sigma(s)\notag\\ & \|\sigma\|_{L^\infty}+\|\sigma'\|_{L^\infty}\leq M_0\|\kappa_{E} -\okappa_{E}\|_{L^2(\pa E)}  \leq M_0 \e_0.
\end{align}
Now it is a simple consequence of \eqref{eq:step5-2}, \eqref{eq:step5-4} and \eqref{eq:step5-5.8} to verify that
$\pa E_i$ is as nearly spherical sets over $D_{r_i}(x_i)$, with $|D_{r_i}(x_i)|=|E_i|$, by functions $f_i\in C^{1,1/2}(\pa D_{r_i}(x_i))$ with 
$C^{1,1/2}$ norm convering to zero as $\e_0\to 0$.

\medskip

\noindent{\bf Step 6.} Quantitative Alexandrov Theorem.

We use the quantitative Alexandrov theorem proven in \cite{MoPoSpa} to infer that, if $f_i$ is the parametrization of the component $E_i$,
then
\[
\|f_i\|_{H^1(\pa D_{r_i}(x_i))}^2 \leq C \| \kappa_{E_i} - \okappa_{E_i} \|_{L^2(\pa E_i)}^2.
\]
Recall that $r_i$ is such that $|E_i|=  |D_{r_i}(x_i)|$. By the area formula, see e.g. \cite[(1.3)]{MoPoSpa}, and a simple linearization we infer that 
\[
0 \leq P(E_i) - P(D_{r_i}(x_i))\leq C \|f_i\|_{H^1(\pa D_{r_i}(x_i))}^2.
\]
Summing over the connected components  yields
\beq
\label{eq:step6-1}
\begin{split}
\| \kappa_{E} - \okappa_{E} \|_{L^2(\pa E)}^2 &\geq \sum_{i=1}^d \| \kappa_{E_i} - \okappa_{E_i} \|_{L^2(\pa E_i)}^2
\geq c\sum_{i=1}^d \|f_i\|_{H^1(\pa D_{r_i}(x_i))}^2\\
&\geq c \,  \big| \sum_{i=1}^d P(E_i) - P(D_{r_i}(x_i))\big|.
\end{split}
\eeq

\medskip

\noindent{\bf Step 7.} Conclusion.

Let $r>0$ be such that the disk  $D_r(x_i)$ has  area  $|D_r(x_i)| = m/d$ , where $m = |E|$. In other words 
$\sum_{i=1}^d P(D_r(x_i)) = 2\pi\,r d = P_d$.  Recall that the disks $D_{r_i}(x_i)$ are defined such that $|D_{r_i}(x_i)| = |E_i|$ for every component $E_i$ and 
thus 
\begin{equation}\label{eqagg}
\sum_i^d r_i^2 = d r^2.
\end{equation}  
Recall also that by the previous estimates it holds $m_0 \leq r_i, r,d \leq M_0$. 
By  \eqref{eq:step5.5.5} and \eqref{eq:step5-5.6} we infer that 
\begin{equation}\label{eq:step7-1}
D_{r_i-M_0 \|\kappa_{E} -\okappa_{E}\|_{L^2(\pa E)}}(x_i) \subset E_i \subset 
D_{r_i +M_0 \|\kappa_{E} -\okappa_{E}\|_{L^2(\pa E)}}(x_i).
\end{equation}
and therefore
\begin{equation}\label{eq:step7-2}
|r-r_i|\leq M_0 \|\kappa_{E} -\okappa_{E}\|_{L^2(\pa E)}.
\end{equation}

Thus, by simple algebra, by \eqref{eqagg} and by \eqref{eq:step5-5.6}, if $d>1$ we deduce
\[
\begin{split}
\big|P_d  -  \sum_{i=1}^d  P(D_{r_i}(x_i))\big| &= 2 \pi \big|d \, r-  \sum_{i=1}^d r_i \big|   =2 \pi \left| \sqrt{d} \left(\sum_{i=1}^d r_i^2 \right)^{\frac12} -  \sum_{i=1}^d r_i \right| \\
&\leq M_0\left( d \sum_{i=1}^d r_i^2  - \left(\sum_{i=1}^d r_i \right)^{2} \right) \\
&= M_0  \sum_{1\leq i<j\leq d} (r_i-r_j)^2\\
&\leq C  \sum_{i=1}^d (r_i-r)^2 \stackrel{\eqref{eq:step7-2}}{\leq} M_0\|\kappa_{E} -\okappa_{E}\|_{L^2(\pa E)}^2.  
\end{split}
\]
Hence, the inequality \eqref{eq:geoineq} then follows by combining the above estimate with \eqref{eq:step6-1}. 
Finally, by the very same argument of step 5 and by \eqref{eq:step7-2} we deduce that the connected components $E_i$ can be parametrized as nearly spherical sets over the discs $D_r(x_i)$.
\end{proof}

Proposition \ref{prop:2DAle} immediately implies the sharp geometric inequality in the plane stated in Theorem \ref{cor:2DAle}. 

\begin{proof}[\textbf{Proof of Theorem \ref{cor:2DAle}}]
Let $\e_0>0$ be from Proposition \ref{prop:2DAle}. If $\|\kappa_E -\bar \kappa_E \|_{L^2(\pa E)} \leq \e_0$ then the inequality holds by  Proposition \ref{prop:2DAle}. If  $\|\kappa_E -\bar \kappa_E \|_{L^2(\pa E)} \geq \e_0$, then the inequality holds trivially as
\[
|P(E) - P_d|  \leq 3M \leq \frac{3M}{\e_0^2} \|\kappa_E -\bar \kappa_E \|_{L^2(\pa E)}^2. 
\]
The inequality \eqref{e.coroAlex} follows similarly. 
\end{proof}

%

\section{The asymptotics of the area preserving curvature flow in the plane}\label{sec:3}

Let us first introduce the setting for the construction of the flat flows. We use the notation from  \cite{MoPoSpa} and refer to \cite{MSS, MoPoSpa} for more detailed introduction.      
We denote the signed distance function by $d_E$ and define it as  
\[
d_E(x) = \dist(x, E) -  \dist(x, \R^2 \setminus E) .
\]
Then clearly $|d_E(x)| = \dist(x, \pa E)$. 

\medskip

We fix the volume $m>0$ and the time step $h>0$, and  given a bounded set $E$ we consider the minimization problem 
\beq
\label{def:min-prob}
\min  \Big{\{} P(F) + \frac{1}{h}\int_F  d_E \, dx + \frac{1}{\sqrt{h}}\big| |F| -m \big| \Big{\}}
\eeq
 and note that the minimizer exists but might not be unique.  We define the dissipation of a set $F$ with respect to a set $E$ as
 \beq
\label{def:distance}
\mathcal{D}(F,E) := \int_{F \Delta E} \dist(x,\pa E)\, dx 
\eeq
and observe that we may write the minimization problem \eqref{def:min-prob} as 
\[
\min  \Big{\{} P(F) + \frac{1}{h}\mathcal{D}(F,E)
 + \frac{1}{\sqrt{h}}\big| |F| -m \big| \Big{\}}.
\]
Let us then recall the construction of the flat flow for the volume preserving mean curvature flow  \eqref{eq:VMCF} from \cite{MSS}. 
Let $E(0) \subset \R^2$ be a bounded  set of finite perimeter which coincides with its Lebesgue representative.  We fix a minimizer of \eqref{def:min-prob}, with $E= E(0)$, denote it by $E_1^{(h)}$ and 
consider its Lebesgue representative.  We construct  the discrete-in-time evolution $\{E_k^{(h)}\}_{k\in \N}$ by recursion such that assuming that $E_k^{(h)}$ 
is defined we set $E_{k+1}^{(h)}$ to be a minimizer of \eqref{def:min-prob} with $E= E_k^{(h)}$.  
By \cite[Lemma 3.1]{MSS}  it holds for all $k =0,1, \dots $
 \beq
\label{eg:energy-compa}
P(E_{k+1}^{(h)}) + \frac{1}{\sqrt{h}}\big||E_{k+1}^{(h)}| -m \big|  + \frac{1}{h}\mathcal{D}(E_{k+1}^{(h)},E_k^{(h)})   \leq P(E_{k}^{(h)}) + \frac{1}{\sqrt{h}}\big||E_{k}^{(h)}| -m \big|  . 
\eeq
Also the set $E_{k+1}^{(h)}$ is $C^{2,\alpha}$-regular and satisfies the Euler-Lagrange equation
 \beq
\label{eg:Euler-Lag}
\frac{d_{E_k^{(h)}}}{h} = - \kappa_{E_{k+1}^{(h)}} + \lambda_{k+1}^{(h)} \qquad \text{on }\, \pa E_{k+1}^{(h)},
\eeq
in the classical sense, where $\lambda_{k+1}^{(h)}$ is the Lagrange multiplier due to the volume penalization. Finally we define the approximative flat flow  $\{E^{(h)}(t)\}_{t \geq 0}$ by setting
\[
E^{(h)}(t) = E_k^{(h)} \qquad \text{for }\, t \in [kh, (k+1) h).
\]
We define a flat flow solution of \eqref{eq:VMCF}  to be any  family of sets $\{E(t)\}_{t \geq 0}$ which is a cluster point of $\{E^{(h)}(t)\}_{t \geq 0}$, i.e., 
\[
E^{(h_n)}(t) \to E(t) \quad \text{as } \, h_n \to 0 \quad \text{in } \, L^1 \quad \text{for almost every }\, t >0 .
\]
By \cite[Theorem 2.2]{MSS} there exists a flat flow starting from $E(0)$ such that $P(E(t)) \leq P(E(0))$ and $|E(t)| = m$ for every $t \geq 0$.

We are interested in the long time behavior of the flow. To this aim we need two technical lemmas. The first lemma is  algebraic.

\begin{lemma}\label{an1}
Let $K\in \N$ and $\{a_k\}_{k\in \{1,\ldots, K\}}$ be a sequence of non-negative numbers and let $\mathcal{I}\subset\{1, \ldots, K\}$.
Assume  that there exists $c>1$ such that 
\[
\sum_{k=i}^{K}a_k\leq ca_i
\]
for every $i\in\{1,\ldots, K\}\setminus \mathcal{I}$.
Then,
$$
\sum_{k=i+1}^{K}a_k\leq \Bigl(1-\frac1c\Bigr)^{i-|\mathcal{I}|}S
$$
for every $i\in \{1,\ldots, K\}$, where $S:=\sum_{k=1}^{K}a_k$ and $|\mathcal{I}|$ denotes the cardinality of $\mathcal{I}$.
\end{lemma}
\begin{proof}
Set $F(i):=\sum_{k=i}^{K}a_k$ and note that  by assumption $F(i)\leq c(F(i)-F(i+1))$ for every $i\in\{1,\ldots, K\}\setminus \mathcal{I}$. 
Hence, we have
$$
F(i+1)\leq 
\begin{cases}
 \Bigl(1-\frac1c\Bigr)F(i) & \textup{if } i \not\in \mathcal{I},\\
 F(i) & \textup{if } i \in \mathcal{I}.
\end{cases}
$$
By iterating the previous estimate (note that at least $K-|\mathcal{I}|$ times the first instance must hold), we conclude.
\end{proof}

The second lemma is in the spirit of Ekeland variational principle.

\begin{lemma}\label{lem:eke}
Let $d \in \N$ and $D_r(x_1), \dots, D_r(x_d)$ be disjoint disks and denote $F =  \bigcup_{i=1}^d D_r(x_i)$. Then there is a constant $C$, which depends only on $d$ and $r$, such that for every set of finite perimeter $E \subset \R^2$ it holds 
\[
P(F) \leq P(E) + C |E\Delta F|^{\frac13}.
\]
\end{lemma}

\begin{proof}
 Let us fix a set $E$ and let $\rho \leq r/10$ be a positive number which choice will be clear later. We begin by constructing a set $F_\rho$ of class $C^{1,1}$, which contains the union of disks $F \subset F_\rho$,  satisfies interior and exterior ball condition with radius $\rho$ and 
\beq
\label{eq:eke1}
P(F) \leq P(F_\rho) +C \sqrt{\rho} , \quad \text{and }\quad |F \Delta F_\rho| \leq C \rho^{\frac32}.  
\eeq
Let $x_1, \dots, x_d$ be the centerpoints of the disks. If  it holds  $|x_i-x_j| > 2r +  \rho$ for every $i \neq j$ we simply choose $F_\rho = F$. If $|x_i-x_j| \leq 2r +  \rho$ for some $i \neq j$ we connect the disks $D_r(x_i)$ and $D_r(x_j)$ with a thin neck around the midpoint $(x_i +x_j)/2$ as follows. We first enlarge the disks by $\rho$ and consider the union $\tilde F_\rho^{ij} := D_{r+\rho}(x_i) \cup D_{r+ \rho}(x_j)$, which overlap around the midpoint  $(x_i +x_j)/2$ . We then decrease the union back by $\rho$ and define  
\[
F_{\rho}^{i,j} = \{ x \in \R^2 : \dist(x,\R^2 \setminus  \tilde F_\rho^{ij}) >\rho \} .
\]
Since $|x_i-x_j| \leq 2r +  \rho$,  the set $F_{\rho}^{i,j}$ is  connected and contains the disks $D_r(x_i)$ and $D_r(x_j)$. The part of the boundary of $F_{\rho}^{i,j} $, which is not contained in $\bar D_r(x_i) \cup \bar D_r(x_j)$, consists of two arcs, see Figure \ref{fig}. In particular, the set $F_{\rho}^{i,j}$  satisfies interior and exterior ball condition with radius $\rho$. We repeat the same construction for all disks  $D_r(x_i)$ and $D_r(x_j)$  which are close to each other in the sense that $|x_i-x_j| \leq 2r +  \rho$, and obtain $F_\rho$ which satisfies \eqref{eq:eke1}. 

\begin{figure}
\begin{tikzpicture}
\clip (-5, -3.3) rectangle (5, 3.3);
\begin{scope}[yscale=0.8,xscale=0.8]

\fill[draw=black, thin, fill=black!10!white] (3,0) circle (2.7); 
\fill[draw=black, thin, fill=black!10!white]   (-3,0) circle (2.7);

\filldraw[fill=black] (3,0) circle (.05);
\draw (3,0.2) node {$x_j$}; 

\filldraw[fill=black] (-3,0) circle (.05);
\draw (-3,0.2) node {$x_i$};

\draw   (-1.85,0.414)+(10:1.5) arc (200: 340: 0.4);

\begin{scope}[yscale=-1,xscale=1]
\draw   (-1.85,0.414)+(10:1.5) arc (200: 340: 0.4);
\end{scope}

\end{scope}                      
\end{tikzpicture}
\caption{If two disks are close to each other, we connect them with a neck given by two arcs}\label{fig}
\end{figure}
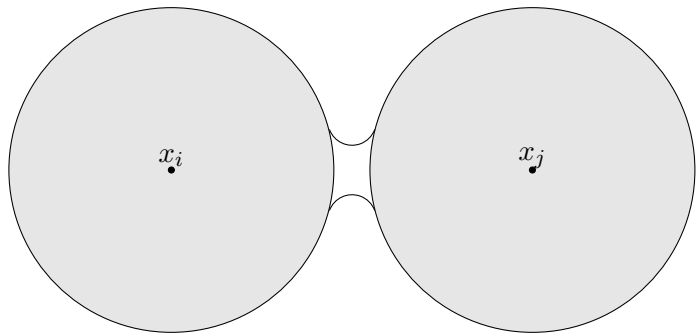

The rest of the proof  follows from standard calibration argument (see e.g. \cite[Proof of Theorem 4.3]{AFM}) and we only give the sketch of the argument.  We construct a vector field $X \in C^{1,1}(\R^2, \R^2)$ such that 
\[
X(x) = \nabla d_{F_\rho}(x) \zeta (x)
\]
where $0 \leq \zeta \leq 1$ is a smooth cut-off function such that $\zeta(x) = 1$ for $|d_{F_\rho}(x)| \leq \rho/4$, $\zeta(x)= 0$ for $|d_{F_\rho}(x)| \geq \rho/2$ and $|\nabla \zeta | \leq C/\rho$. In particular, it holds $|X|\leq 1$ in $\R^2$ and $X = \nu_{F_\rho}$ on $\pa F_\rho$. Moreover, since $F_\rho$ satisfies interior and exterior ball condition with radius $\rho$ it holds $|\Delta   d_{F_\rho}(x)| \leq C/\rho$ for $|d_{F_\rho}(x)| \leq \rho/2$. Therefore by the divergence theorem 
\[
P(F_\rho) - P(E) \leq  \int_{F_\rho \Delta E} |\div(X)|\, dx \leq \frac{C}{\rho} |F_\rho \Delta E| .
\]

We combine the above inequality with \eqref{eq:eke1} and deduce
\[
P(F) \leq P(E) + \frac{C}{\rho} |E \Delta F| + C \sqrt{\rho}.
\]
Choosing $\rho =\min\{ |E \Delta F|^{\frac23} , r/10\}$ yields the claim. 
\end{proof}

We may now give the proof of the convergence of the area-preserving mean curvature flow.   

\begin{proof}[\textbf{Proof of Theorem \ref{thm2}}]
Let $\{E(t)\}_{t\geq0}$ be an area-preserving flat flow and let $\{E^{(h_n)}(t)\}_{t\geq0}$ be an approximate
flow converging to $E(t)$.
Set
\[
f_n(t) = P(E^{(h_n)}(t)) + \frac{1}{\sqrt{h_n}} \big||E^{(h_n)}(t)|-m\big|.
\]
By \eqref{eg:energy-compa} the $f_n$'s are monotone non-increasing functions which are  bounded by $P(E(0))$. Therefore, by Helly's selection theorem,
up to passing to a further subsequence (not relabeled), the functions $f_n$'s converge pointwise to some non-increasing function $f_\infty:[0,+\infty)\to \R$.
Set $F_\infty= \lim_{t\to +\infty}f_\infty(t)$. In what follows we also set 
\[
v^{(h_n)}_t=\frac{d_{E^{(h_n)}_k}}{h_n}\,, \qquad \text{where }\, k = \left\lfloor\frac t{h_n}\right\rfloor-1
\]
the approximate velocity of the approximate flow at time $t$. Moreover, $C$ will denote a positive constant, which may change from line to line and might depend on the flat flow itself (but not on $h_n$ nor on the discrete step of the minimizing movements).

We divide the proof in two cases.

\medskip

\textbf{Case 1:} There exists $d\in \N\setminus\{0\}$ such that either $P_d<F_\infty<P_{d+1}$ or $F_\infty=P_d$ and $f_\infty(t) >P_d$ for every $t\in [0,+\infty)$.

In this case, there exists $\bar t >0$ such that, for every $T>\bar t$ there exist $\bar n\in \N\setminus\{0\}$ such that
\begin{gather}
\label{expo0.9}
P_d\leq f_n(t) < P_{d+1}\qquad \text{and}\qquad P_{d+1} -f_n(t) \geq \frac{P_{d+1}-F_\infty}{2}=:\delta_0
\end{gather}
for every $n\geq \bar n$ and   $t \in [ \bar t, T]$.
Set $\mathcal{I}^{(h_n)}=\left\{i\in \big\{\lfloor\frac{\bar t}{h_n}\rfloor, \ldots, \lfloor\frac{T}{h_n}\rfloor\big\}: |E^{(h_n)}_i|\neq m\right\}$.
By \cite[Cor. 3.10]{MSS} there exists a constant $C_T>0$ such that 
\beq \label{eq:MSS0} 
|\mathcal{I}^{(h_n)}|\leq C_T
\eeq
for $n$ sufficiently large.
For every $i \not\in \mathcal{I}^{(h_n)}$ we have by iterating \eqref{eg:energy-compa} and using \eqref{expo0.9} 
\[
\frac{1}{h_n}\sum_{k=i+1}^{\lfloor\frac{T}{h_n}\rfloor} \mathcal{D}(E^{(h_n)}_{k},E^{(h_n)}_{k-1}) \leq P(E_i^{(h_n)}) - P(E^{(h_n)}_{\lfloor\frac{T}{h_n}\rfloor}) \leq P(E_i^{(h_n)})  -P_d.
\]
Then by \eqref{e.coroAlex} and by the Euler-Lagrange equation \eqref{eg:Euler-Lag} 
\begin{equation}
\label{eq:th}
\begin{split}
\frac{1}{h_n}\sum_{k=i+1}^{\lfloor\frac{T}{h_n}\rfloor} \mathcal{D}(E^{(h_n)}_{k},E^{(h_n)}_{k-1}) &\leq P(E_i^{(h_n)}) - P_d \leq C_0 \|\kappa_{E_i^{(h_n)}}- \okappa_{E_i^{(h_n)}}\|_{L^2(\pa E_i^{(h_n)}) }^2 \\
&\leq C_0 \|\kappa_{E_i^{(h_n)}}- \lambda_i^{(h_n)}\|_{L^2(\pa E_i^{(h_n)}) }^2= \frac{C_0}{h_n^2} \int_{\pa E_i^{(h_n)}} d_{E_{i-1}^{(h_n)}}^2\, d\H^1 .
\end{split}
\end{equation}
In \cite{MSS} it is proven (formula after (3.25)) that 
\beq \label{eq:MSS1}
 \int_{\pa E_i^{(h_n)}} d_{E_{i-1}^{(h_n)}}^2\, d\H^1  \leq C\mathcal{D}(E^{(h_n)}_{i},E^{(h_n)}_{i-1}) .
\eeq
Therefore from \eqref{eq:th} we conclude 
\[
\sum_{k=i+1}^{\lfloor\frac{T}{h_n}\rfloor} \mathcal{D}(E^{(h_n)}_{k},E^{(h_n)}_{k-1}) \leq \frac{C_0'}{h_n}\mathcal{D}(E^{(h_n)}_{i},E^{(h_n)}_{i-1}) .
\]

Setting $a^{(h_n)}_k= h_n^{-1}\mathcal{D}(E^{(h_n)}_{k},E^{(h_n)}_{k-1})$ we have that 
for every $i \in \big\{\lfloor\frac{\bar t}{h_n}\rfloor, \ldots, \lfloor\frac{T}{h_n}\rfloor\big\}\setminus\mathcal{I}^{(h_n)}$ it holds 
\[
\sum_{k=i}^{\lfloor\frac{T}{h_n}\rfloor} a_k^{(h_n)}
 \leq \frac{C'_0+h_n}{h_n} a_i^{(h_n)}\leq \frac{2C'_0}{h_n} a_i^{(h_n)}.
 \]
Moreover it holds by \eqref{eg:energy-compa} $\sum_{k=1}^\infty a_k \leq P(E(0)) \leq M$.  By Lemma \ref{an1} we infer that
\[
\sum_{k=i+1}^{\lfloor\frac{T}{h_n}\rfloor}a_k^{(h_n)}\leq M \left(1-\frac{h_n}{2C'_0}\right)^{i-C_T- \frac{\bar t}{h_n}} 
\qquad \text{for all } \,  i= \lfloor\frac{\bar t}{h_n}\rfloor, \ldots, \lfloor\frac{T}{h_n}\rfloor.
\]
In other words for every $t\in [\bar t, T]$ we have 
\beq\label{expo1}
\sum_{k=\lfloor\frac{t}{h_n}\rfloor+1}^{\lfloor\frac{T}{h_n}\rfloor}
h_n^{-1}\mathcal{D}(E^{(h_n)}_{k},E^{(h_n)}_{k-1})
\leq M \left(1-\frac{h_n}{2C'_0}\right)^{\lfloor\frac{t}{h_n}\rfloor-C_T-\frac{\bar t}{h_n}}\leq C e^{-\frac{t}{2C'_0}}
\eeq
for $h_n\leq h_0(T)$.

By \cite[Proposition 3.4]{MSS} it holds 
\[
 |E^{(h_n)}_i\Delta E^{(h_n)}_{i-1}| \leq C\ell P(E^{(h_n)}_i)
+ \frac{C}{\ell} \int_{E^{(h_n)}_i\Delta E^{(h_n)}_{i-1}} |d_{E^{(h_n)}_{i-1}}| \, dx 
\]
for all $\ell \leq \frac1C \sqrt{h_n}$.  Therefore, by the inequality above and by \eqref{expo1} we infer    that for every $\bar t \leq t < s \leq T$ we have
\[
\begin{split}
|E^{(h_n)}(t)\Delta E^{(h_n)}(s)|  & = \sum_{i=\lfloor\frac{t}{h_n}\rfloor +1 }^{\lfloor\frac{s}{h_n}\rfloor} |E^{(h_n)}_i\Delta E^{(h_n)}_{i-1}|\\
&\leq C \sum_{i=\lfloor\frac{t}{h_n}\rfloor +1 }^{\lfloor\frac{s}{h_n}\rfloor} \left(\ell P(E^{(h_n)}_i)
+ \frac{1}{\ell} \int_{E^{(h_n)}_i\Delta E^{(h_n)}_{i-1}} |d_{E^{(h_n)}_{i-1}}|\, dx \right)\\
& \leq C P(E(0)) \ell \frac{s-t}{h_n} +\frac{C}{\ell}\sum_{i=\lfloor\frac{t}{h_n}\rfloor +1 }^{\lfloor\frac{s}{h_n}\rfloor}\mathcal{D}(E^{(h_n)}_{i},E^{(h_n)}_{i-1})\\
&\leq C M \ell \frac{s-t}{h_n} M+\frac{C \, h_n}{\ell} e^{-\frac{t}{2C'_0}},
\end{split} 
\]
for all $\ell \leq \frac1C \sqrt{h_n}$ and $h_n\leq h_0$.
In particular, choosing $\ell= \frac{h_n}{e^{\alpha t}}$ with $\alpha = \frac{1}{4C'_0}$ and $s\leq t+1$, we have
\[
|E^{(h_n)}(t)\Delta E^{(h_n)}(s)| \leq CM e^{-\frac{t}{4C'_0}}.
\]
Passing to the limit as $h_n\to 0$, we get
\beq  \label{expo1.5}
|E(t)\Delta E(s)| \leq  CM  e^{-\frac{t}{4C'_0}} \qquad \text{for all} \, \, \bar t \leq t \leq s\leq t+1.
\eeq
Hence,  we deduce that $E(t)$ converges exponentially fast to a set of finite perimeter $E_\infty$ in $L^1$ and $|E_\infty| = m$.

We now show that the  limiting set $E_\infty$ is the union of disjoint open disks with the same radius.
Denote by $S_\infty$ the countable set of discontinuity points of $f_\infty$ and note that for any $t\in (0,+\infty)\setminus S_\infty$ and any sequence $t_n\to t$ we have $f_n(t_n)\to f_\infty(t)$.

Fix $t\geq \bar t$, $0<\alpha<\frac1{2C'_0}$, and an open set $A(t)$ such that $S_\infty\cap[t, t+e^{-\alpha t}]\subset A(t)\subset [t, t+e^{-\alpha t}]$ and $|A(t)|\leq e^{-\alpha' t}$, with $\alpha'>\alpha$. By \eqref{eq:MSS1} and  \eqref{expo1}  we have  
\beq\label{expo2}
\begin{split}
&\int_{[t, t+e^{-\alpha t}]\setminus A(t)} \Big(\int_{\pa E^{(h_n)}(s)} (v^{(h_n)}_s)^2 d\H^1\Big) ds \leq
\frac{1}{h_n} \sum_{i=\lfloor \frac{t}{h_n}\rfloor}^{\lfloor \frac{t+e^{-\alpha t}}{h_n}\rfloor}\int_{\pa E_i^{(h_n)}} d_{E_{i-1}^{(h_n)}}^2 \, d \H^1 \\
&\leq C
\sum_{i=\lfloor \frac{t}{h_n}\rfloor}^{\lfloor \frac{t+e^{-\alpha t}}{h_n}\rfloor}\frac{1}{h_n}\mathcal{D}(E^{(h_n)}_{i},E^{(h_n)}_{i-1}) \leq C   e^{-\frac{t}{2C'_0}}\,,
\end{split}
\eeq
for $n$ sufficiently large. By possibly increasing $\bar t$  we have  $|[t, t+e^{-\alpha t}]\setminus A(t)|>\frac12 e^{-\alpha t}$ for $t\geq \bar t$. Moreover  by \eqref{eq:MSS0}  it holds 
\[
|\{s\in[t, t+e^{-\alpha t}]:\, |E^{(h_n)}(s)|\neq m \}|\to 0\,, \quad \text{as } \, n \to \infty.  
\]
Then  by \eqref{expo2} and by  the mean value theorem there exists $s_n \in [t, t+e^{-\alpha t}]\setminus A(t)$ such that 
\beq\label{expo3}
\|\kappa_{E^{(h_n)}(s_n)}-\okappa_{E^{(h_n)}(s_n)} \|_{L^2(\pa E^{(h_n)}(s_n))}^2\leq \int_{\pa E^{(h_n)}(s_n)} (v^{(h_n)}_{s_n})^2 d\H^1 \leq Ce^{-\big(\frac{1}{2C'_0}-\alpha\big)t}\,,
\eeq
 $|E^{(h_n)}(s_n)|= m$, and thus, in particular, $f_n(s_n)=P(E^{(h_n)}(s_n))$.
From Proposition \ref{prop:2DAle} and \eqref{expo3},  we infer that, for $t \geq \tilde t$, where $\tilde t$ is sufficiently large, $E^{(h_n)}(s_n)$ is diffeomorphic to a union of $d$ disjoint disks
and 
\beq\label{expo3.5}
|P(E^{(h_n)}(s_n)) - P_d| \leq Ce^{-\big(\frac{1}{2C'_0}-\alpha\big)t}.
\eeq
In particular, passing to the limit in $h_n\to 0$ (up to a further not relabelled subsequence, if needed), there exists $s_t \in [t, t+ e^{-\alpha t}]\setminus A(t)$ such that $s_n\to s_t$ and thus $E^{(h_n)}(s_n)\to E(s_t)$  in $L^1$ and  $P(E^{(h_n)}(s_n))=f_n(s_n)\to f_\infty(s_t)$.
In fact, by the uniform $C^{1,\frac12}$-bounds provided by \eqref{expo3} and Proposition \ref{prop:2DAle}  we deduce  that 
$P(E^{(h_n)}(s_n))\to P(E(s_t))$ and thus $f_\infty(s_t)=P(E(s_t))$, and that   $E(s_t)$ is the union of $d$ nearly spherical sets parametrized over $d$ disjoint open disks $D_r(x_i(t))$, $i=1,\dots d$ of volume $m/d$,   with $C^{1, \frac12}$-norm
of the parametrizations (exponentially) small. In particular, setting $F(t):=\cup_{i=1}^d D_r(x_i(t))$, we have that $\sup_{x\in E(s_t)\Delta F(t)}\dist(x, \pa F(t))$ decays exponentially to zero  as $t\to+\infty$, $E_\infty$ is a union of $d$ disjoint open disks of volume $m/d$,  and $F(t)\to E_\infty$ in the Hausdorff sense exponentially fast.

Summarizing,  and recalling also the first inequality in \eqref{expo0.9} and \eqref{expo3.5}, we have shown  that for every $t$ sufficiently large, there exists  $s_t\in [t,t+e^{-\alpha t}] $ such that $E(s_t)$ is the union of $d$ disjoint nearly spherical sets parametrized over the disjoint open disks of $E_\infty$ 
and
\beq\label{expo4}
 P_d\leq  f_\infty (s_t)=P(E(s_t))\leq P_d+Ce^{-(1/C-\alpha)t}\,, \qquad \sup_{x\in  E(s_t)\Delta E_\infty}\dist(x, \pa E_\infty)\leq C e^{-\frac{t}{C}}\,,
\eeq
for a suitable constant $C>1$. 


From the first inequality in \eqref{expo4} and by the monotonicity of $f_\infty$  we obtain  for all $s$ sufficiently large that by choosing $t$ such that  $s =t+e^{-\alpha t}$ it holds 
\[
P(E(s))\leq f_\infty(s) \leq f_\infty(s_t)  \leq P_d+Ce^{-(1/C-\alpha)(s-e^{-\alpha t})}\leq
P_d+Ce^{-\frac{(1/C-\alpha)s}2}\,.
\]
On the other hand, by  Lemma \ref{lem:eke} and \eqref{expo1.5} we obtain
\[
P_d \leq P(E(t)) + C|E(t) \Delta E_\infty|^{\frac13} \leq P(E(t)) + C' e^{-\frac{t}{12C_0'}}\,.
\]
Hence, we have the exponential convergence of the perimeters in \eqref{expo0}. 

The first part of the inequality in \eqref{expo0} follows from the second inequality in \eqref{expo4} and  from \cite[Lemma 4.3]{JN}.

\medskip

\textbf{Case 2}: There exist $d\in \N\setminus\{0\}$ and $\bar t> 0$ such that $F_\infty=P_d=f_\infty(t)$ for every $t\geq \bar t$.

In this case, using the monotonicity of the functions $f_n$'s, we deduce that for every $T>\bar t$ the functions $f_n$ converge uniformly to $f_\infty\equiv F_\infty$
in $[\bar t, T]$.
In particular, using that
\[
\frac{1}{h_n} \mathcal{D}(E^{(h_n)}_k, E^{(h_n)}_{k-1}) \leq f_n((k-1)h_n) - f_n(k h_n),
\]
we deduce that for every $t\in [\bar t + h_n, T]$ we have
\[
\sum_{k=\lfloor\frac{t}{h_n}\rfloor+1}^{\lfloor\frac{T}{h_n}\rfloor}
h_n^{-1}\mathcal{D}(E^{(h_n)}_{k},E^{(h_n)}_{k-1})
\leq f_n\Big(\lfloor\frac{t}{h_n}\rfloor h_n\Big) - f_n\Big(\lfloor\frac{T}{h_n}\rfloor h_n\Big)=:b_n
\to F_\infty - F_\infty=0 \qquad \textup{as } h_n\to 0.
\]
Arguing as above, for every $\bar t +h_n \leq t < s \leq T$, we get
\begin{align*}
|E^{(h_n)}(t)\Delta E^{(h_n)}(s)|  \leq C \ell \frac{s-t}{h_n} P(E(0))+\frac{C}{\ell}\sum_{i=\lfloor\frac{t}{h_n}\rfloor +1 }^{\lfloor\frac{s}{h_n}\rfloor}\mathcal{D}(E^{(h_n)}_{i},E^{(h_n)}_{i-1}),
\end{align*} 
for all $\ell \leq \frac1C \sqrt{h_n}$ and, choosing $\ell= \sqrt{b_n} h_n$, we conclude that
\begin{align*}
|E^{(h_n)}(t)\Delta E^{(h_n)}(s)|  \leq C \sqrt{b_n} (s-t) P(E(0))+C \sqrt{b_n} \to 0,
\end{align*}
that is $E(t) = E(s)$ for every $\bar t<t<s<T$.

The final part of the proof consists in showing that the limiting set $E_\infty$ is the union of disjoint open disks with the same radius. We have 
\[
\int_t^T \int_{\pa E^{(h_n)}(t)} (v_t^{(h_n)})^2 d\H^1 =
\frac{1}{h_n} \sum_{i=\lfloor \frac{t}{h_n}\rfloor}^{\lfloor \frac{T}{h_n}\rfloor}\int_{\pa E_i^{(h_n)}} d_{E_{i-1}^{(h_n)}}^2 =
\sum_{i=\lfloor \frac{t}{h_n}\rfloor}^{\lfloor \frac{T}{h_n}\rfloor}\frac{1}{h_n}\mathcal{D}(E^{(h_n)}_{i},E^{(h_n)}_{i-1}) = o(1).
\]
By the mean value theorem, for ever $T$ sufficiently large there exists $t_n \in [T,T+1]$ such that 
\[
\|\kappa_{E^{(h_n)}(t_n)}-\okappa_{E^{(h_n)}(t_n)} \|_{L^2(\pa E^{(h_n)}(t_n))}^2\leq \int_{\pa E^{(h_n)}(t_n)} (v_{t_n}^{(h_n)})^2 d\H^1 =o(1).
\]
As before, by Proposition \ref{prop:2DAle} the sets $E^{(h_n)}(t_n)$ are nearly spherical and converges to union of $d$ disjoint open balls. From here the conclusion follows.
\end{proof}

\section{The asymptotics of the 2D Mullins-Sekerka flow}\label{sec:4}

Let us first construct a flat flow solution for the Mullins-Sekerka flow in the 2-dimensional flat torus. The construction in the case of bounded domain is due to Luckhaus and St\"urzenhecker   \cite{LS} and the same construction can be applied to the periodic setting with obvious changes. 
We denote the perimeter of a set $E $ in the flat torus $\T^2$ by $P_{\T^2}(E) $ and recall that it is defined as
\[
P_{\T^2}(E) := \sup \Big{\{} \int_E \div X \, dx : X \in C^1(\T^2,\R^2) , \, \| X\|_{L^\infty} \leq 1 \Big{\}}.
\]
Here $X \in C^1(\T^2,\R^2)$ means that the $\Z^2$-periodic extension of $X$ to $\R^2$ is continuously differentiable. For a given set of finite perimeter $E \subset \T^2$, with $|E|=m$, we consider the minimization problem 
\beq \label{def:min-prob2}
\min \Big{\{} P_{\T^2}(F) + \frac{h}{2} \int_{\T^2} |\nabla U_{F,E}|^2 \, dx  :\quad \text{with }\, |F| = |E| = m \Big{\}},  
\eeq
where the function  $ U_{F,E}\in H^1(\T^2)$ is the solution of 
\beq \label{eq:potential1}
-\Delta U_{F,E} = \frac{1}{h} \left( \chi_F - \chi_E\right)
\eeq
with zero average.  As proven in \cite{LS, Rog} there exists a minimizer for \eqref{def:min-prob2}, but it might not be unique.  Concerning the regularity of the minimizers we may  argue as in \cite[Theorem 2.8.]{AFM} (see also \cite[Proposition 2.2]{MoPoSpa}) to deduce that the minimizing set  $F$ is $C^{3,\alpha}$-regular. Let us briefly sketch the argument. First,  we may replace the volume constraint in \eqref{def:min-prob2} by volume penalization as in \cite{AFM, EF} and conclude that the minimizer  is a $\Lambda$-minimizer of the perimeter. This implies that the minimizer  is $C^{1,\alpha}$-regular and satisfies the associated Euler-Lagrange equation
\[
U_{F,E}  = -\kappa_{F} + \lambda \qquad \text{on }\, \pa F
\]
in a weak sense, where $\lambda$ is the Lagrange multiplier. Since $U_{F,E}$ is the solution of \eqref{eq:potential1},  by standard elliptic regularity it holds $U_{F,E} \in C^{1,\alpha}(\T^2)$. Then by  the Euler-Lagrange equation we deduce that $F$ is in fact $C^{3,\alpha}$-regular and the Euler-Lagrange equation holds in the classical sense.

Let us denote 
 \beq
\label{def:distance2}
\mathfrak{D}(F,E) := \int_{\T^2} |\nabla U_{F,E}|^2 \, dx 
\eeq
where $U_{F,E}$ is defined in \eqref{eq:potential1}. We define the $H^{-1}$-norm by duality  as 
\[
\|f\|_{H^{-1}(\T^2)} :=  \sup \Big\{ \int_{\T^2} \varphi \, f \, dx : \|\nabla \varphi\|_{L^2(\T^2)} \leq 1\Big\} .
\]
Then by integrating  \eqref{eq:potential1}  by parts yields  
\beq
\label{def:Hmenouno-bound}
\|\chi_F - \chi_E\|_{H^{-1}(\T^2)}^2  \leq  h^2 \, \| \nabla U_{F,E} \|_{L^2(\T^2)}^2 = h^2 \, \mathfrak{D}(F,E).
\eeq

We fix the time step $h >0$ and our initial set $E(0) \subset \T^2$ and let $E_1^{(h)}$ be a minimizer of \eqref{def:min-prob2} with $E(0) =E$.   We construct the discrete-in-time evolution $(E_k^{(h)})_{k\in \N}$ as before by induction such that, assuming that $E_k^{(h)}$  is defined, we set $E_{k+1}^{(h)}$ to be a minimizer of \eqref{def:min-prob2} with $E= E_k^{(h)}$ and denote the associated potential for short  by $U_{k+1}^{(h)}$, which is the solution of  
\beq \label{eq:potential2}
-\Delta U_{k+1}^{(h)} = \frac{1}{h} \left( \chi_{E_{k+1}^{(h)}} - \chi_{E_{k}^{(h)}}\right)
\eeq 
with zero average. The Euler-Lagrange equation now reads as 
\beq \label{eq:Euler-Lagr2}
U_{k+1}^{(h)} = -\kappa_{E_{k+1}^{(h)}} + \lambda_{k+1}^{(h)} \qquad \text{on }\, \pa E_{k+1}^{(h)}.
\eeq
By a direct energy comparison (formula (3.6) in \cite{Rog}) we obtain 
\beq \label{eq:energy-compa2}
P_{\T^2}(E_{k+1}^{(h)})  +\frac{h}{2}  \mathfrak{D}(E_{k+1}^{(h)}, E_{k}^{(h)}) \leq P_{\T^2}(E_{k}^{(h)})  ,
\eeq
where $ \mathfrak{D}(E_{k+1}^{(h)}, E_{k}^{(h)})$ is defined in \eqref{def:distance2}.

As before we define the approximative flat flow  $\{E^{(h)}(t)\}_{t \geq 0}$ by setting
\[
E^{(h)}(t) = E_k^{(h)} \qquad \text{for }\, t \in [kh, (k+1) h)
\]
and  we call a \emph{flat flow solution}  of \eqref{eq:Mull-Sek}   any cluster point  $\{E(t)\}_{t \geq 0}$ of $\{E^{(h)}(t)\}_{t \geq 0}$, as $h\to 0$;  i.e., 
\[
E^{(h_n)}(t) \to E(t) \quad \text{ in $L^1$ for almost every }\, t >0 \text{ and for some }h_n\to 0.
\]
Arguing exactly as in  \cite[Proposition 3.1]{Rog} we may conclude that there exists a flat flow starting from $E(0)$ such that $P_{\T^2}(E(t)) \leq P_{\T^2}(E(0))$, $|E(t)| = |E(0)|$ for every $t \geq 0$ and $\{E(t)\}_{t\geq 0}$ satisfies the equation \eqref{eq:Mull-Sek} in a weak sense.

To proceed, we need the analogue  of Proposition~\ref{prop:2DAle}  for the Mullins-Sekerka flow. To this aim we first prove the following lemma, which is similar to \cite[Lemma 2.1]{Sch}. 
\begin{lemma}
\label{lem:schatzle}
 Let $E\subset\T^2$ be a set of class $C^3$, with $|E| \leq \frac12$ and  $P_{\T^2}(E) <2 $,  and let $u_E \in C^1(\T^2)$  be a function with zero average such that  $\|\nabla u_E\|_{L^2(\T^2)} \leq M$ and 
\beq \label{eq:boundary}
\kappa_E = -u_E +\lambda \quad \text{on } \,  \pa E \quad \text{ for some } \lambda\in\R .
\eeq
Then it holds 
\beq \label{eq:density-est}
\sup_{x \in \T^2, \rho >0} \frac{\H^1(\pa E \cap D_\rho(x))}{\rho} \leq K,
\eeq
where the constant $K>0$ depends only on $|E|$ and $M$.
\end{lemma}
\begin{proof}
We note that by \eqref{eq:boundary}   for every $X \in C^1(\T^2; \R^2)$  it holds
\beq \label{eq:1stVar}
\int_{\pa E} \div_\tau X \, d \H^1 = \int_{E} \div \big((-u_E + \lambda)X\big) \, dx. 
\eeq
Therefore the statement follows from \cite[Lemma 2.1]{Sch} once we bound the Lagrange multiplier $\lambda \in \R$.  To this aim, and for future purpose, we show that there is $\delta>0$ such that  every component $E_i$ of $E$ is contained in a cube $Q_{1-\delta}(x_i) := (1-\delta)^2 +\{x_i\}$  for some $x_i$. 

Let us first show that every component $\Gamma_i$  of the boundary $\pa E$  divides the torus $\T^2$ in two components and thus it is the boundary of a set.   Indeed, if this is not the case then necessarily  $\H^1(\Gamma_i) \geq 1$. Since $\Gamma_i$ is not a boundary of a set then $\pa E$ must have another component, say  $\Gamma_j$, such that $\H^1(\Gamma_j)\geq 1$. But this implies $P_{\T^2}(E) = \H^1(\pa E) \geq 2$, which contradicts  the assumption $P_{\T^2}(E)< 2$.   

Let us next show that $\Gamma_i$ is contained in a cube $Q_{1-\delta}(x_i)$  for some $x_i$. Let $\pi_1: \T^2 \to \T$ be the projection onto the $x_1$-axis i.e., $\pi_1(x_1, x_2) = x_1$. Then we deduce from  $\H^1(\Gamma_i) <2$ and from the fact that $\Gamma_i$ is the boundary of a set that $\H^1(\pi_1(\Gamma_i))<1$. Similarly it holds  $\H^1(\pi_2(\Gamma_i))<1$, where $\pi_2$ is the projection onto the $x_2$-axis. This implies that $\Gamma_i \subset Q_{1-\delta}(x_i)$ for some $\delta>0$ and $x_i$. Let us from now on  denote  the set enclosed by $\Gamma_i$ which is inside the cube $Q_{1-\delta}(x_i)$ by $F_i$.

Let  $\Gamma_1, \dots, \Gamma_n$ be the components of the boundary $\pa E$ which enclose the sets  $F_1, \dots, F_n$.  Let us show that 
\beq \label{eq:ann1}
E \subset \bigcup_{i=1}^n F_i.
\eeq 
Since $F_i \subset Q_{1-\delta}(x_i)$  we have by the Isoperimetric Inequality $2\sqrt{\pi |F_i|} \leq \H^1(\Gamma_i)$. Therefore by the assumption on the perimeter,  $P_{\T^2}(E)<2$, we have 
\[
4 \pi \big| \bigcup_{i=1}^n F_i \big| \leq 4 \pi \sum_{i=1}^n |F_i| \leq \sum_{i=1}^n \H^1(\Gamma_i)^2 \leq \Big(  \sum_{i=1}^n \H^1(\Gamma_i) \Big)^2  \leq P_{\T^2}(E)^2 < 4.
\] 
Therefore $\big| \bigcup_{i=1}^n F_i \big| < \frac{1}{\pi}< \frac13$. Since, $|E| \leq \frac12$ then necessarily $E \subset  \bigcup_{i=1}^n F_i$.

We conclude from \eqref{eq:ann1} that a component $E_j$ of $E$ is contained   in $F_i$ for some $i$. Therefore since $F_i \subset Q_{1-\delta}(x_i)$, then also $E_j \subset Q_{1-\delta}(x_i)$.

We may finally bound the Lagrange multiplier in \eqref{eq:1stVar}  by a standard argument. Indeed, let $E_j$ be a component of $E$. Since $E_j \subset Q_{1-\delta}(x_j)$ we may define $X \in C_0^1\big(Q_{1-\delta/3}(x_j)\big)$ such that   $X(x) = x$ in $E_j$ and $X(x) = 0$ in $E \setminus E_j$. We apply  \eqref{eq:1stVar}  with this choice of $X$ and have  
\[
\begin{split}
P_{\T^2}(E_j) =  \int_{\pa E_j} \div_\tau x \, d \H^1  &= \int_{E_j} \div \big((-u_E + \lambda)x\big) \, dx\\
&= -\int_{E_j} \div ( u_E \, x) \, dx +2  \lambda |E_j|. 
\end{split}
\]
We have $\big| \int_{E_j} \div ( u_E \, x) \, dx \big| \leq C\|u_E\|_{H^1(E_j)}$. By repeating the argument for every component we obtain by the Poincar\'e inequality 
\[
|\lambda| |E| \leq  P_{\T^2}(E) + C\|u_E\|_{H^1(E)}\leq   P_{\T^2}(E) + C\|u_E\|_{H^1(\T^2)} \leq P_{\T^2}(E) + C\| \nabla u_E\|_{L^2(\T^2)}. 
\]
This yields the required bound on the Lagrange multiplier. 
\end{proof}

We also recall the result  by Meyers-Ziemer \cite[Theorem 4.7]{MZ} which implies that if $E$ satisfies \eqref{eq:density-est} then for every $\varphi \in C^1(\T^2)$ it holds 
\beq \label{eq:MZ}
\big| \int_{\pa E} \varphi \, d \H^1 \big| \leq C \|\varphi\|_{W^{1,1}(\T^2)}\,,
\eeq
with $C$ depending on $K$ (and thus on $|E|$ and $M$).

We are now ready to state and prove the analogue of Proposition~\ref{prop:2DAle}, which is suited for the Mullins-Sekerka flow. 

\begin{proposition} \label{prop:2DAle2}
 Let $E\subset\T^2$ be a set of class $C^3$, with $|E|=m \leq \frac12$ and  $P_{\T^2}(E)< 2$, and let $u_E \in C^1(\T^2)$ be a function with zero average such that 
\[
\kappa_E = -u_E +\lambda \quad \text{on } \,  \pa E
\]
 for some $\lambda \in \R$. Then, there exist $\e_0=\e_0(m)\in (0,1)$ and $C_0=C_0(m)>1$ such that  if 
\[
\|\nabla u_E\|_{L^2(\T^2)} \leq \e_0
\]
then $E$ is diffeomorphic to a union of $d$ disjoint disks $D_1$, \dots, $D_{d}$ with equal areas $m/{d}$ and $\dist(D_i, D_j)>0$ for $i\neq j$. Moreover,   
\[
|P_{\T^2}(E)- P_{d}|\leq C_0\|\nabla u_E\|_{L^2(\T^2)}^2 
\]
and for $\e_0$ sufficiently small the boundary of every connected component
of the set $E$ can be parametrized as a normal graph over a one of the disc $D_i$ 
with $C^{1,\frac12}$ norm of the parametrization vanishing as $\e_0\to 0$.
\end{proposition}
We note that we need the assumption $P_{\T^2}(E)< 2$ to exclude the case when  $E$ is a strip or a union of strips.  

\begin{proof}
We recall that the argument in the proof of Lemma \ref{lem:schatzle} implies that every component $E_i$ of $E$ is contained in a cube $Q_{1-\delta}(x_i)$ for some $x_i$.  By Lemma \ref{lem:schatzle} we can apply  \eqref{eq:MZ} with $\varphi = u_E^2 $ and obtain
\[
 \int_{\pa E} u_E^2\, d \H^1 \leq C \|u_E^2\|_{W^{1,1}(\T^2)} \leq C \|u_E\|_{H^1(\T^2)}^2 \leq C\| \nabla u_E\|_{L^2(\T^2)}^2, 
\]
where the last inequality follows from Poincar\'e inequality. Since $u_E$ satisfies \eqref{eq:boundary} we deduce by the assumption $\|\nabla u_E\|_{L^2(\T^2)} \leq \e_0$ that 
\[
 \int_{\pa E} |\kappa_E - \okappa_E|^2 \, d \H^1\leq   \int_{\pa E} |\kappa_E - \lambda|^2 \, d \H^1 =  \int_{\pa E} u_E^2\, d \H^1\leq C\| \nabla u_E\|_{L^2(\T^2)}^2 \leq C \e_0^2. 
\]
Hence, the claim follows from Proposition \ref{prop:2DAle}. 
\end{proof}

Proposition \ref{prop:2DAle2} immediately implies the following corollary.

\begin{corollary}\label{coro:2Dale2}
 Let $E\subset\T^2$ be a set of class $C^3$, with $|E|=m \leq \frac12$ and  $P_{\T^2}(E)< 2$ and let $u_E \in C^1(\T^2)$ be a function with zero average such that $\kappa_E = -u_E +\lambda$ on $\pa E$ for some $\lambda \in \R$.  If $\delta_0>0$ and $d \in  \N$, are such that $P_{d} \leq P(E) \leq P_{d+1} - \delta_0$, then it holds 
\[
P(E) - P_d  \leq C_0 \|\nabla u_E\|_{L^2(\T^2)}^2 
\]
for $C_0= C_0(m,\delta_0)$.
\end{corollary}

We also need the following lemma which is essentially a restatement of \cite[Lemma 3.1]{LS}. The proof can also be found in  \cite[Lemma 2]{CL}, but we recall it for the reader's convenience. 
\begin{lemma}
\label{lem:luckhaus}
Let $\varphi \in  \text{BV}(\T^2)$. There is a constants $C>1$ and $\rho_0 >0$ such for all $\rho \leq \rho_0$ it holds  
\[
\|\varphi \|_{L^1(\T^2)} \leq C \rho \,  \|\varphi\|_{BV(\T^2)}+C \rho^{-1} \|\varphi\|_{H^{-1}(\T^2)}. 
\]
\end{lemma}
\begin{proof}
 Let us fix $\rho >0$ and let $\eta_{\rho}(x) = \rho^{-2}\eta(\tfrac{x}{\rho})$ be the standard mollifier. Then we write 
\[
\|\varphi\|_{L^1(\T^2)} \leq \int_{\T^2} |\varphi - \varphi * \eta_{\rho}| \, dx +  \int_{\T^2} | \varphi * \eta_{\rho}| \, dx .
\]
Let us first bound the second term on the RHS. Since $\|\eta_{\rho}\|_{H^{1}(\T^2)} \leq C/\rho$ we obtain by the definition of the $H^{-1}$-norm
\[
\begin{split}
\int_{\T^2} | \varphi * \eta_{\rho}| \, dx = \int_{\T^2} \big| \int_{\T^2} \varphi(y) \eta_{\rho}(y-x) \, dy \big| dx \leq \|\varphi\|_{H^{-1}(\T^2)} \|\eta_{\rho}\|_{H^{1}(\T^2)} \leq   C \rho^{-1}\, \|\varphi\|_{H^{-1}(\T^2)}. 
\end{split}
\]
We bound the first term by change of variables
\[
\begin{split}
    \int_{\T^2} |\varphi - \varphi * \eta_{\rho}| \, dx &=  \int_{\T^2} \big| \int_{\T^2} \big(\varphi(x) - \varphi(x+\rho y)\big) \eta(y) \, dy \big| \, dx\\ 
    &= \int_{\T^2} \big| \int_{\T^2} \int_{0}^\rho - \frac{\pa}{\pa \tau}\varphi(x+\tau y) \eta(y) \, d \tau dy \big| \, dx\\
    &\leq C \rho \, \|\varphi\|_{BV(\T^2)}\, .
\end{split}
\]
\end{proof}

We are ready to prove the convergence of the Mullisk-Sekerka flow in the flat torus $\T^2$.

\begin{proof}[\textbf{Proof of Theorem \ref{thm3}}]
The proof is similar to the proof of Theorem \ref{thm2} but we highlight the main differences.  Let $\{E(t)\}_{t\geq0}$ be a flat flow for the Mullins-Sekerka and let $\{E^{(h_n)}(t)\}_{t\geq0}$ be an approximate flow converging to $E(t)$. Since $\{\T^2\setminus E(t)\}_{t\geq0}$ is a flat flow starting from $\T^2\setminus E(0)$, by replacing $E(0)$ with its complement in $\T^2$ if needed, we may assume without loss of generality that $|E(0)|\leq\frac12$. We will  show that in this case the limiting set is a finite union of disjoint open discs with equal radii.

Arguing as before we deduce that by \eqref{eq:energy-compa2} the functions 
\[
f_n(t) = P_{\T^2}(E^{(h_n)}(t))
\]
are monotone non-increasing with  $f_n(t) <2$ and (possibly up to a further unrelabelled subsequence)  converge pointwise to a non-increasing function $f_\infty:[0,+\infty)\to \R$.
Set $F_\infty= \lim_{t\to +\infty}f_\infty(t)$.  Again we divide the proof in two cases.

\textbf{Case 1}: There exists $d\in \N\setminus\{0\}$ such that either $P_d<F_\infty<P_{d+1}$, or $F_\infty=P_d$ and $f_\infty(t) >P_d$ for every $t\in [0,+\infty)$.
In this case, there exists $\bar t \geq 1$ such that, for every $T>\bar t$ there exist $\bar n\in \N\setminus\{0\}$ such that
\begin{gather}
\label{eq:thm3-1}
P_d\leq f_n(t) < P_{d+1}\qquad \text{and}\qquad P_{d+1} -f_n(t) \geq \frac{P_{d+1}-F_\infty}{2}=:\delta_0
\end{gather}
for every $n\geq \bar n$ and   $t \in [ \bar t, T]$.
By summing \eqref{eq:energy-compa2}  and using \eqref{eq:thm3-1} we obtain 
for every $i \in \big\{\lfloor\frac{\bar t}{h_n}\rfloor, \ldots, \lfloor\frac{T}{h_n}\rfloor\big\}$ that 
\beq\label{eq:thm3-2}
\frac{h_n}{2}\sum_{k=i+1}^{\lfloor\frac{T}{h_n}\rfloor} \mathfrak{D}(E^{(h_n)}_{k},E^{(h_n)}_{k-1}) \leq P_{\T^2}(E_i^{(h_n)}) - P_{\T^2}\big(E^{(h_n)}_{\lfloor\frac{T}{h_n}\rfloor}\big) \leq P_{\T^2}(E_i^{(h_n)})  -P_d,
\eeq
where $\mathfrak{D}(E^{(h_n)}_{k},E^{(h_n)}_{k-1}) $ is defined in \eqref{def:distance2}.   Then by \eqref{eq:thm3-2} and by Corollary  \ref{coro:2Dale2}  it holds
\[
\frac{h_n}{2}\sum_{k=i+1}^{\lfloor\frac{T}{h_n}\rfloor}  \mathfrak{D}(E^{(h_n)}_{k},E^{(h_n)}_{k-1}) \leq P_{\T^2}(E_i^{(h_n)}) - P_d \leq  C_0\|\nabla U_i^{(h)}\|_{L^2(\T^2)}^2= C_0 \mathfrak{D}(E^{(h_n)}_{i},E^{(h_n)}_{i-1}).
\]
Therefore we conclude 
\[
\sum_{k=i+1}^{\lfloor\frac{T}{h_n}\rfloor} \mathfrak{D}(E^{(h_n)}_{k},E^{(h_n)}_{k-1}) \leq \frac{2C_0}{h_n}\mathfrak{D}(E^{(h_n)}_{i},E^{(h_n)}_{i-1}) .
\]
Setting $a^{(h_n)}_k=  h_n \, \mathfrak{D}(E^{(h_n)}_{k},E^{(h_n)}_{k-1})$
we have that 
for every $i \in \big\{\lfloor\frac{\bar t}{h_n}\rfloor, \ldots, \lfloor\frac{T}{h_n}\rfloor\big\}$ it holds 
\[
\sum_{k=i}^{\lfloor\frac{T}{h_n}\rfloor} a_k^{(h_n)}
 \leq \frac{2C_0+h_n}{h_n} a_i^{(h_n)}\leq \frac{3C_0}{h_n} a_i^{(h_n)}
 \]
and by applying  \eqref{eq:thm3-2} with $i = \lfloor\frac{\bar t}{h_n}\rfloor$ yields
\[
\sum_{k=\lfloor\frac{\bar t}{h_n}\rfloor +1}^{\lfloor\frac{T}{h_n}\rfloor} a_k^{(h_n)} \leq 
P_{\T^2}\Big(E_{\lfloor\frac{\bar t}{h_n}\rfloor}^{(h_n)}\Big) \leq 
P_{\T^2}(E(0)) < 2.
\]
Therefore  Lemma \ref{an1}, with $\mathcal{I}$ being the empty set this time,   implies
\[
\sum_{k=i+1}^{\lfloor\frac{T}{h_n}\rfloor}a_k^{(h_n)}\leq 2 \left(1-\frac{h_n}{3C_0}\right)^{i- \frac{\bar t}{h_n}} 
\qquad \text{for all }\, i= \lfloor\frac{\bar t}{h_n}\rfloor, \ldots, \lfloor\frac{T}{h_n}\rfloor.
\]
In other words for every $t\in [\bar t, T]$ we have 
\[
\sum_{k=\lfloor\frac{t}{h_n}\rfloor+1}^{\lfloor\frac{T}{h_n}\rfloor}
h_n\, \mathfrak{D}(E^{(h_n)}_{k},E^{(h_n)}_{k-1})
\leq  2 \left(1-\frac{h_n}{3C_0}\right)^{\lfloor\frac{t}{h_n}\rfloor - \frac{\bar t}{h_n}}\leq  C e^{-\frac{t}{3C_0}}
\]
for $h_n\leq h_0(T)$. Then,  by \eqref{def:Hmenouno-bound} and by the above inequality we have that for $\bar t \leq t <s \leq T$ with $s \leq t+1$ it holds
\beq\label{eq:thm3-3}
\begin{split}
\|\chi_{E^{(h_n)}(s)} &- \chi_{E^{(h_n)}(t)}\|_{H^{-1}(\T^2)} \leq \sum_{k=\lfloor\frac{t}{h_n}\rfloor+1}^{\lfloor\frac{s}{h_n}\rfloor} \|\chi_{E^{(h_n)}_{k}} - \chi_{E^{(h_n)}_{k-1}}\|_{H^{-1}(\T^2)}  \\ 
&\leq \frac{\sqrt{s-t}}{\sqrt{h_n}} \Big( \sum_{k=\lfloor\frac{t}{h_n}\rfloor+1}^{\lfloor\frac{T}{h_n}\rfloor} \|\chi_{E^{(h_n)}_{k}} - \chi_{E^{(h_n)}_{k-1}}\|_{H^{-1}(\T^2)}^2\Big)^{\frac12} \\
&\leq  \frac{1}{\sqrt{h_n}} \Big( \sum_{k=\lfloor\frac{t}{h_n}\rfloor+1}^{\lfloor\frac{T}{h_n}\rfloor} h_n^2 \, \mathfrak{D}(E^{(h_n)}_{k},E^{(h_n)}_{k-1})\Big)^{\frac12}\\
&\leq C\, e^{-\frac{t}{6C_0}},
\end{split}
\eeq
when $h_n\leq h_0(T)$. 

Recall for all $t >0$ it holds $\|\chi_{E^{(h_n)}(t)}\|_{\text{BV}(\T^2)}\leq \|\chi_{E(0)}\|_{\text{BV}(\T^2)} \leq 3$. We use Lemma \ref{lem:luckhaus}  and \eqref{eq:thm3-3} to deduce  
\[
\begin{split}
\|\chi_{E^{(h_n)}(s)} &- \chi_{E^{(h_n)}(t)}\|_{L^1(\T^2)} \\
&\leq C \e  \, \|\chi_{E^{(h_n)}(s)} - \chi_{E^{(h_n)}(t)}\|_{\text{BV}(\T^2)} + C \e^{-1} \|\chi_{E^{(h_n)}(s)} - \chi_{E^{(h_n)}(t)}\|_{H^{-1}(\T^2)} \\
&\leq C \e  + C \e^{-1} e^{-\frac{t}{6C_0}}. 
\end{split}
\]
Choosing $\e =  e^{-\frac{t}{12C_0}}$ yields 
\[
\|\chi_{E^{(h_n)}(s)} - \chi_{E^{(h_n)}(t)} \|_{L^1(\T^2)}  \leq C e^{-\frac{t}{12C_0}}.
\]
 Letting $h_n \to 0$ we obtain  for the limit flow 
\[
|E(s) \Delta E(t)|   \leq C \,  e^{-\frac{t}{12C_0}}.
\]
From here we conclude that $E(t)$ converges to a set of finite perimeter $E_\infty$ exponentially fast. 

We may characterize the limit set $E_\infty$ as a disjoint union of open disks $D_r(x_1), \dots, D_r(x_d)$ thanks to Proposition \ref{prop:2DAle2} by arguing as in the proof of Theorem~\ref{thm2}. Similarly, we  obtain the convergence of the perimeters.  We leave the details for the reader.  

 Also the argument for the Case 2, when there exist $d\in \N\setminus\{0\}$ and $\bar t > 0$ such that $F_\infty=P_d=f_\infty(t)$ for every $t\geq \bar t$, follows by the same argument as in  the proof of Theorem~\ref{thm2}. 
\end{proof}

\section*{Acknowledgments}
V.~J.~was supported by the Academy of Finland grant 314227.  E.~S.~has been supported by the ERC-STG grant 759229 {\sc HiCoS}. M.~M. was supported by the University of Parma FIL grant `Regularity, Nonlinear Potential Theory and related topics''.


\begin{thebibliography}{99}

\bibitem{AFJM}\textsc{E. Acerbi, N. Fusco, V. Julin, M. Morini}, \emph{Nonlinear stability results for the modified Mullins-Sekerka and the surface diffusion flow.} J. Differential Geom. \textbf{113} (2019), 1--53. 

\bibitem{AFM}
\textsc{E. Acerbi, N. Fusco, M. Morini}. 
\emph{Minimality via second variation for a nonlocal isoperimetric problem.} Comm.
Math. Phys. \textbf{322} (2013),  515--557.

\bibitem{AlBaCh}
 \textsc{N.D. Alikakos, P.W. Bates \& X. Chen}, \emph{Convergence of the Cahn-Hilliard equation to the Hele-Shaw model.} Arch. Rational Mech. Anal. \textbf{128} (1994),  165--205.

\bibitem{ATW}
\textsc{F. Almgren, J. Taylor, L. Wang}.  
\emph{Curvature-driven flows:a variational approach}. SIAM J. Optim.  \textbf{31}
(1993),  387--438.

\bibitem{AFP} \textsc{L. Ambrosio, N. Fusco, D. Pallara}  \emph{Functions of bounded variation and free discontinuity problems.} Oxford Mathematical Monographs. The Clarendon Press, Oxford University Press, New York, 2000.


\bibitem{BCCN} 
\textsc{G. Bellettini , V. Caselles , A. Chambolle , M. Novaga}, 
\emph{The volume preserving crystalline mean curvatureflow of convex sets in $\R^N$.}
{J. Math. Pure Appl.} {\bf 92} (2009), 499--527.

\bibitem{CRCT95} \textsc{W. Carter, A. Roosen, J. Cahn, and J. Taylor}. \emph{Shape evolution by surface diffusion and surface attachment limited kinetics on completely faceted surfaces.} Acta Metallurgica et Materialia {\bf 43} (1995),  4309--4323. 

\bibitem{CL}
\textsc{A. Chambolle, T. Laux}.
\emph{Mullins-Sekerka as the Wasserstein flow of the perimeter.} Proc. Amer. Math. Soc. \textbf{149} (2021), 2943--2956.

\bibitem{Chen} \textsc{X. Chen}, \emph{The Hele-Shaw problem and area-preserving curve-shortening motions.} Arch. Rational Mech. Anal. \textbf{123} (1993), 117--151.

\bibitem{Ci} 
\textsc{G. Ciraolo},
\emph{Quantitative estimates for almost constant mean curvature hypersurfaces}. Boll. Unione Mat. Ital. 14 (2021), 137--150. 


\bibitem{CM}
\textsc{G. Ciraolo, F. Maggi},
 \emph{On the shape of compact hypersurfaces with almost-constant mean curvature}. Comm.
Pure Appl. Math. \textbf{70} (2017), 665--716.

\bibitem{CV}
\textsc{G. Ciraolo, L. Vezzoni}. 
\emph{A sharp quantitative version of Alexandrov’s theorem via the method of moving planes}. J. Eur. Math. Soc. (JEMS) \textbf{20} (2018), 261--299.


\bibitem{DeGKu}\textsc{D. De Gennaro, A. Kubin}, \emph{Long time behaviour of the discrete volume preserving mean curvature flow in the flat torus.} Preprint 2021.

\bibitem{DM}
\textsc{M.  Delgadino, F.  Maggi}, 
\emph{Alexandrov's theorem revisited}. Anal. PDE \textbf{12} (2019),  1613--1642.


\bibitem{DMMN}
\textsc{M. Delgadino, F. Maggi, C. Mihaila, R. Neumayer}, 
\emph{Bubbling with $L^2$-almost constant mean curvature and an Alexandrov-type theorem for crystals}. Arch. Ration. Mech. Anal. \textbf{230} (2018),   1131--1177.

\bibitem{RKS}
\textsc{A. De Rosa, S. Kolasinski,  M. Santilli}
\emph{Uniqueness of critical points of the anisotropic isoperimetric problem for finite perimeter sets}. Arch. Ration. Mech. Anal. \textbf{238} (2020), 1157--1198.

\bibitem{EI} 
\textsc{J. Escher, K. Ito}, \emph{Some dynamic properties of volume preserving curvature driven flows.} Math. Ann. \textbf{333} (2005), 213--230.

\bibitem{ES}
\textsc{J. Escher, G. Simonett}. 
\emph{The volume preserving mean curvature flow near spheres}. Proc. Am. Math. Soc. \textbf{126} (1998), 2789--2796. 

\bibitem{ES98} \textsc{J. Escher; G. Simonett},
\emph{A center manifold analysis for the Mullins-Sekerka model.} 
J. Differential Equations \textbf{143} (1998), 267--292. 

\bibitem{EF}
\textsc{L. Esposito, N. Fusco}. 
\emph{A remark on a free interface problem with volume constraint}. J. Convex Anal. \textbf{18} (2011), 417--426.


\bibitem{FJM}
\textsc{N. Fusco, V. Julin,  M. Morini}.
\emph{Stationary sets and asymptotic behavior  of the  mean curvature flow with forcing in the plane}, to appear in  J. Geom. Anal.

\bibitem{FMP}
\textsc{N. Fusco, F. Maggi, A. Pratelli},
\emph{The sharp quantitative isoperimetric inequality}. Ann. of Math. 168, 941--980 (2008).


\bibitem{GarRau}
\textsc{H. Garcke, M. Rauchecker}, \emph{Stability analysis for stationary solutions of the Mullins-Sekerka flow with boundary contact.} Preprint 2019.



\bibitem{Hui}
\textsc{G. Huisken}. 
\emph{The volume preserving mean curvature flow.} J. Rein. Angew. Math \textbf{382} (1987), 35--48.


\bibitem{JN}
\textsc{V. Julin, J. Niinikoski},
 \emph{Quantitative Alexandrov Theorem and asymptotic behavior of the volume preserving mean curvature flow}. Preprint 2020. 

\bibitem{KK}
\textsc{I. Kim, D. Kwon}. 
\emph{Volume preserving mean curvature flow for star-shaped sets.} Comm. Partial Diffenretial Equations \textbf{45} (2020), 414--455.


\bibitem{KM}
\textsc{B. Krummel, F. Maggi}, \emph{Isoperimetry with upper mean curvature bounds and sharp stability estimates.} {Calc. Var. Partial Differ. Equ.} {\bf 56} (2017), Article n. 53.

\bibitem{LS}
\textsc{S. Luckhaus, T. St\"urzenhecker}, 
 \emph{Implicit time discretization for the mean curvature flow equation}. Calc. Var. PDEs, \textbf{3} (1995), 253--271.


\bibitem{MaggiBook} \textsc{F. Maggi},
\emph{Sets of finite perimeter and geometric variational problems. An introduction to geometric measure theory.} Cambridge Studies in Advanced Mathematics, 135. Cambridge University Press, Cambridge, 2012.

\bibitem{M}
\textsc{U.~F.~Mayer,}
\emph{A singular example for the average mean curvature flow.}
{Experimental Mathematics} {\bf 10} (2001), 103--107.

\bibitem{MaSim}
\textsc{U.~F.~Mayer, G.~Simonett},
\emph{Self-intersections for the surface diffusion and the volume-preserving mean curvature flow.} {Differential and Integral Equations} {\bf 13} (2000), 1189--1199.


\bibitem{MZ}
\textsc{N. Meyers, W. P. Ziemer}, 
\emph{Integral inequalities of Poincar\'e and Wirtinger type
for BV-functions}. Amer. J. of Math. \textbf{99} (1977) 1345--1360.


\bibitem{MoPoSpa}
\textsc{M. Morini, M. Ponsiglione, E. Spadaro}, 
\emph{Long time behaviour of discrete volume preserving mean curvature flows.} Preprint 2020. https://arxiv.org/abs/2004.04799

\bibitem{MuSe13} \textsc{L. Mugnai, C. Seis},
\emph{On the coarsening rates for attachment-limited kinetics.}
{SIAM J. Math. Anal.} {\bf 45} (2013), 324--344.


\bibitem{MSS} 
\textsc{L. Mugnai, C. Seis, E. Spadaro},
\emph{Global solutions to the volume-preserving mean-curvature flow.}
 Calc. Var. and PDEs {\bf 55}(2016), Article n.~18.

\bibitem{MS} \textsc{W.W Mullins,  R.F. Sekerka},  \emph{ Morphological stability of a particle growing by diffusion or heat flow,} Fundamental contributions to the continuum theory of evolving phase interfaces in solids, 75--81, Springer, Berlin, 1999.

\bibitem{Joonas}
\textsc{J. Niinikoski}, 
\emph{Volume preserving mean curvature flows near strictly stable sets in flat torus.} J. Differential Equations \textbf{276} (2021), 149--186.

\bibitem{Pego}
\textsc{R.L. Pego}, \emph{Front migration in the nonlinear Cahn-Hilliard equation.} Proc. Roy. Soc. London Ser. A \textbf{422} (1989),  261--278.

\bibitem{Rog} 
\textsc{M. R\"oger},
\emph{Existence of weak solutions for the Mullins-Sekerka flow}. 
 Siam J. Math. Anal.  \textbf{37} (2005), 291--301. 

\bibitem{RS} 
\textsc{M. R\"oger, R. Sch\"atzle},
\emph{Control of the isoperimetric deficit by the Willmore deficit}. Analysis (Munich) \textbf{32} (2012),  1--7. 


\bibitem{Sch} 
\textsc{R. Sch\"atzle}, 
\emph{Hypersurfaces with mean curvature given by an ambient Sobolev function},
 J. Differential Geom. \textbf{58} (2001), 371--420.

\bibitem{TW72} \textsc{L. A. Tarshis,  J. L. Walker and M. F. X. Gigliotti}, \emph{
Solidification.}{ Annual Review of Materials Science} {\bf 2}(1972), 181--216.

\bibitem{W61} \textsc{C. Wagner},
\emph{Theorie der Alterung von Niederschl\"agen durch Um{l\"o}sen(Ostwald-Reifung).}
{ Zeitschrift  f\"ur  Elektrochemie,  Berichte  der  Bunsengesellschaft f\"ur physikalische Chemie} {\bf 65} (1961), 581--591.



\end{thebibliography}
\end{document}